%% file: ad0.tex
\documentclass{amsart}
\usepackage{graphics}
\usepackage{color}
\usepackage{epsfig}
\usepackage{url}
\usepackage{amsmath}
\usepackage{amssymb}
\newtheorem{theor}{Theorem}
\newtheorem{lem}{Lemma}
\newtheorem{cor}{Corollary}
\newtheorem{proposition}{Proposition}

\theoremstyle{remark}

\newtheorem*{rem}{Remark}

\newcommand{\RR}{\ensuremath{\mathbb{R}}}

\newcommand{\ZZ}{\ensuremath{\mathbb{Z}}}

\newcommand{\Eseven}{\mathsf{E}_7}
\newcommand{\Esix}{\mathsf{E}_6}
\newcommand{\Efive}{\mathsf{E}_5}

\newcommand{\KnP}{\mathsf{K}_{n+1}}
\newcommand{\An}{\mathsf{A}_{n}}

\DeclareMathOperator{\Stab}{Stab}
\DeclareMathOperator{\Aut}{Aut}

\DeclareMathOperator{\Lin}{Lin}
\DeclareMathOperator{\conv}{conv}
\DeclareMathOperator{\diplo}{diplo}

\DeclareMathOperator{\Fan}{Fan}

\begin{document}

\author{Mathieu Dutour Sikiri\'c}
\address{Mathieu Dutour Sikiri\'c, Rudjer Boskovi\'c Institute, Bijenicka 54, 10000 Zagreb, Croatia}
\email{mdsikir@irb.hr}
\thanks{The first author is supported by the Croatian Ministry of Science, Education and Sport under contract 098-0982705-2707}

\author{Viacheslav Grishukhin}
\address{Viacheslav Grishukhin, CEMI Russian Academy of Sciences,
Nakhimovskii prosp. 47 117418 Moscow, Russia}
\email{grishuhn@cemi.rssi.ru}

\author{Alexander Magazinov}
\address{Alexander Magazinov, Steklov Mathematical Institute of RAS, Gubkina str. 8, 119991 Moscow, Russia}
\email{magazinov-al@yandex.ru}
\thanks{The third author is supported by the Russian government project 11.G34.31.0053 and RFBR grant 11-01-00633}

\markleft{On the sum of a parallelotope and a zonotope}

\title{On the sum of the Voronoi polytope of a lattice with a zonotope}

\date{}

\begin{abstract}
A {\em parallelotope} $P$ is a polytope that admits a facet-to-facet tiling of space by translation copies of $P$ along a lattice.
The Voronoi cell $P_V(L)$ of a lattice $L$ is an example of a
parallelotope.
A parallelotope can be uniquely decomposed as the Minkowski sum of a zone
closed parallelotope $P$ and a zonotope $Z(U)$, where $U$ is the set of vectors
used to generate the zonotope.
In this paper we consider the related question: When is the Minkowski
sum of a general parallelotope and a zonotope $P+Z(U)$ a parallelotope?
We give two necessary conditions and show that the vectors $U$ have to
be {\em free}. Given a set $U$ of free vectors, we give several methods
for checking if $P + Z(U)$ is a parallelotope.
Using this we classify such zonotopes for some highly symmetric lattices.

In the case of the root lattice $\Esix$, it is possible to give a more
geometric description of the admissible sets of vectors $U$.
We found that the set of admissible vectors, called free vectors,
is described by the well-known configuration of $27$ lines in a cubic.
Based on a detailed study of the geometry of $P_V(\Esix)$,
we give a simple characterization of the configurations of vectors $U$
such that $P_V(\Esix) + Z(U)$ is a parallelotope.
The enumeration yields $10$ maximal families of vectors, which are
presented by their description as regular matroids.
\end{abstract}

\maketitle

\section*{Acknowledgments}
This article is dedicated to the memory of the
late Evgenii Baranovskii (died in March 2012).

\section{Introduction}\label{sect:Intro}

Let $L = \ZZ v_1+\dots + \ZZ v_n$ be a {\em lattice} of rank $n$ in Euclidean
space given by $n$ independent vectors $(v_i)_{1\leq i\leq n}$ in $\RR^n$.
The {\em Voronoi cell} (or {\em Voronoi polytope}) of $L$ is the convex polytope
\begin{equation*}
P_V(L) = \left\{x \in \RR^n : v^Tx\le\frac{1}{2}v^2\mbox{  for all  }v\in L - \{0\} \right\},
\end{equation*}
where $v^Tx$ is scalar product of vectors $v,x\in\RR^n$.
By ${\mathcal R}$ we denote the set of relevant vectors, i.e. vectors $v\in L$ that determine facet defining inequalities of $P_V(L)$.
The quotient group $L/2L$ allows to determine the relevant vectors. A coset $v + 2L$ with $v\in L$ is called {\em simple} if it has a unique, up to sign, minimal vector. One can prove that ${\mathcal R}$ is formed by the simple cosets.
Above description of the Voronoi polytope $P_V(L)$ is given in an orthonormal basis of $\RR^n$. This basis is self dual.

If we consider vectors of the lattice $L$ in a basis $B$ of $L$, and points $x\in\RR^n$ in the dual basis $B^*$, then we obtain the following affinely equivalent description of $P_V(L)$:
\begin{equation*}
P_V(L,f) = \left\{x \in \RR^n : z^Tx\le f(z)\mbox{  for all  } z\in{\mathcal Z}\right\},
\end{equation*}
where ${\mathcal Z}\subset {\mathbb Z}^n$ is a set of {\em representations} of minimal vectors vectors of $L$ in the basis $B$. Here $f(z)=\frac{1}{2}z^T Dz$ is a quadratic form related to the chosen basis of $L$, and $D$ is Gram matrix of the basis.

For a root lattice $L$ the set $\mathcal R$ is a set of all roots, i.e. of all vectors $v\in L$ of norm $v^2=2$.   Root lattices ${\mathsf A}_n$, ${\mathsf D}_n$, for $n\ge 1$, and ${\mathsf E}_n$, for $n=6,7,8$, are described in many books and papers on lattices.
See, for example, \cite{CS}.

A {\em parallelotope} is a polytope whose translation copies under a lattice $L$ fill the space without gaps and intersections by inner points. For a given lattice $L$, $\{P_V(L)+v\}_{v\in L}$ define a tiling of $\RR^n$ and so $P_V(L)$ is a parallelotope.

We call a {\em facet normal} each vector $p$ orthogonal to a hyperplane supporting a facet $F(p)$, i.e., an $(n-1)$-dimensional face, of an $n$-parallelotope $P$. If $P=P_V(L)$ is a Voronoi polytope defined above, then every its facet normal $p$ is collinear to some lattice vector $s$ such that the facet $F(p)$ 
is shared by the parallelotopes $P_V(L)$ and $P_V(L)+s$. Such lattice vectors are called {\it facet vectors} of $F(p)$.

A parallelotope $P$ is necessarily centrally symmetric and its facets are
necessarily centrally symmetric.
A {\em $k$-belt} of an $n$-dimensional polytope $P$, whose facets are
centrally symmetric, is a family of $k$ facets $F_1,F_2,\dots,F_k$ such
that $F_i\cap F_{i+1}$ and $F_{i}\cap F_{i-1}$ are antipodal
$(n-2)$-dimensional faces in $F_i$ for $1\le i\le k$, where the indexing
$i$ in $F_{i}$ taken modulo $k$. A polytope $P$ is a parallelotope if and
only if the following {\em Venkov conditions} hold \cite{venkov,mcmullen,zong}:
\begin{itemize}
\item $P$ is centrally symmetric;
\item The facets of $P$ are centrally symmetric;
\item The facets of $P$ are organized into $4$- and $6$-belts.
\end{itemize}
A still open conjecture of Voronoi \cite{voronoi} asserts that any
parallelotope is affinely equivalent to a Voronoi polytope.
Voronoi conjecture has been solved up to dimension $5$ \cite{engelE6s}.
For a given set $U$ of vectors, the {\em zonotope} $Z(U)$ is the
Minkowski sum
\begin{equation*}
Z(U) = \sum_{u\in U} z(u),
\end{equation*}
where
\begin{equation}
\label{zu}
z(u)=[-u,u]=\{x\in{\mathbb R}^n:x=\lambda u, \mbox{  }-1\le\lambda\le 1\}
\end{equation}
is a segment of a line spanned by the vector $u$.

Voronoi's conjecture has been proved for zonotopes \cite{dicing,dicing2}.

For a parallelotope $P$, a vector $v$ is called {\em free}
if the Minkowski sum $P+z(v)$ is again a parallelotope. In \cite{Gr} 
free vectors were characterized by the requirement that $z(v)$ is parallel to 
at least one facet of each $6$-belt of $P$, but the argument that this condition 
was both necessary and sufficient was incomplete; a complete proof is provided below in 
Section~\ref{SectionEnumAlgo}.

A parallelotope is called {\em free} or {\em nonfree} depending upon
whether or not it has free vectors; the first known nonfree parallelotope
is $P_V(\Esix^*)$
\cite{engelE6s} and it was later proved that the parallelotopes
$P_V(\mathsf{D}_{2m}^+)$ are also nonfree for $m\geq 4$ \cite{free}.
A parallelotope $P$ is called {\em zone-open} if it can be represented as
a Minkowski sum $P'+z(v)$, where $P'$ is a parallelotope and $v$ is a non-zero vector;
otherwise the parallelotope is called {\em zone-closed}.
Any parallelotope can be uniquely expressed as Minkowski sum of a
zone-closed parallelotope and a zonotope.

A parallelotope $P$ is called {\em finitely free} if there exists a
finite set ${\mathcal F}(P)$ such that any free vector $v$ is collinear to a vector in ${\mathcal F}(P)$.
The parallelotopes that we consider in this paper are finitely free
Voronoi polytopes of highly symmetric lattices. But, for example, 
any vector $v$ is free for the cube $[-1/2, 1/2]^n$, which is 
the Voronoi polytope for the lattice $\ZZ^n$, because the cube has no $6$-belts.

For a given zone-closed parallelotope $P$, we want to find the vector
systems $U$ such that $P + Z(U)$ is still a parallelotope.
Of course, every vector in $U$ has to be free. Another condition arising
from the theory of matroid \cite{BZ,DaG,Gr} is that $U$ has to be {\em unimodular}, i.e. in any basic subset $B\subseteq U$, all vectors of $U$ have integer coordinates in the basis $B$.

The condition that elements of $U$ are free for $P$, and that $U$ is unimodular are 
certainly necessary for $P + Z(U)$ to be a parallelotope. But in general they are not
sufficient, as shown below in the discussion of the Voronoi polytope $P_V(\Esix)$
of the root lattice $\Esix$. If we want to check Venkov conditions, then we have
to determine the faces and the $k$-belts of the sum $P + Z(U)$.

If $G$ is a face of $P$, then we decompose $U$ into
$U_1(G)\cup U_2(G)\cup U_3(G)$.
Let $x\in G$ be an inner point of $G$ and let $\lambda$ be a sufficiently small number. Then one
of the following three cases are true.
\begin{itemize}
\item $u\in U_1(G)$ if $x\pm\lambda u\notin G$ and one of the points $x\pm \lambda u$ is in $P$.
\item $u\in U_2(G)$ if $x\pm \lambda u\in \Lin(G)$.
\item $u\in U_3(G)$ if $x\pm \lambda u\notin G$ and $x\pm\lambda u\notin P$.
\end{itemize}
The vectors in $U_1(G)$ translate $G$ by some vector
$w=\sum_{u\in U_1(G)} \pm u$. The vectors $u\in U_2(G)$ belong
to the vector space $\Lin(G)$ defined by $G$ and extend $G$ to a larger face of the same dimension.
The vectors $u\in U_3(G)$ are {\em strongly transversal} to $G$, that
is $\dim (G+z(u)) = 1 + \dim G$ and $G+z(u)$ is a new face of $P$.

Hence, all faces of $P + Z(U)$ are translated or extended facets of $P$.
So, if one computes all facets and $(n-2)$-dimensional faces of $P + Z(U)$,
then one can determine if it is a parallelotope.
The method is explained in Section \ref{SectionEnumAlgo} and applied
in Section \ref{SectionE6zonotope} (with modifications) to the root lattice $\Esix$ 
and in Section \ref{SectionOtherLattice} to other lattices.

The Voronoi polytope $P_V(\Esix)$ of the root lattice $\Esix$
has only $6$-belts. The edges of $P_V(\Esix)$ are of two types:
$r$-edges, which are parallel and equal by norm to minimal vectors of the
lattice $\Esix$, and $m$-edges, which are parallel and equal by norm
to minimal vectors of the dual lattice $\Esix^*$. The set $\mathcal R$
of minimal vectors of $\Esix$ is the root system $\Esix$.
Roots of $\mathcal R$ are also facet vectors of $P_V(\Esix)$.
One can choose a set ${\mathcal M}$ of $27$ vectors of $\Esix^*$ such
that the pairwise scalar products $q^Tq'$ of $q,q'\in\mathcal M$
take only two values $-\frac{2}{3}$ and $\frac{1}{3}$ and the set of minimal
vectors of  $\Esix^*$ is ${\mathcal M}\cup -{\mathcal M}$.

For a lattice $L$ we define an {\em empty sphere} to be a ball containing
no lattice point in its interior. The convex hull of the lattice points
on the surface of the ball is called a {\em Delaunay polytope}.
The Delaunay polytopes define a tessellation, which is dual
to the one by Voronoi polytopes $P_V(L)$, in particular the vertices of
$P_V(L)$ are circumcenters of Delaunay polytopes (see, for example, \cite{schurmannBook}).
It turns out that the Delaunay polytopes of $\Esix$ are Sch\"afli
polytope $P_{Schl}$ and its antipodal $P^*_{Schl}$.
The configuration of vectors ${\mathcal M}$ is the vertex set of $P_{Schl}$.

We prove that the set of free vectors of $\Esix$ is exactly $\mathcal M$. Then 
we consider subsets $U\subseteq\mathcal M$
and give a necessary and sufficient condition for $P_V(\Esix)+Z(U)$
to be a parallelotope.
Namely, the subset $U$ should
not contain a subset of five vectors $u_i\in\mathcal M$,
$1\le i\le 5$, such that $u_i^Tu_j=\frac{1}{3}$ for $1\le i<j\le 5$.
We find all ten maximal by inclusion feasible subsets
$U\subseteq\mathcal M$ by a computer search. We give a detailed description
of these subsets and regular matroids, represented by these subsets. Finally, we consider how the implemented
methods can be extended to other lattices.

We use Coxeter's notations $\alpha_n$ and $\beta_n$ for regular $n$-dimensional simplices and cross-polytopes respectively.

In Section \ref{sec:gen_results} general results about free vectors
are proved.
In Section \ref{sec:standard_vector} the standard vectors of parallelotopes
$P+z(v)$ are determined.
In Section \ref{SectionEnumAlgo} enumeration algorithms for free-vectors are given.
In Sections \ref{SectionRootlattice} the lattice $\mathsf{E}_6$ is built.
In Sections \ref{SectionSchlafli}, \ref{SectionVoronoi} and \ref{SectionFreedom} the Delaunay polytope, Voronoi polytope and free vectors of $\mathsf{E}_6$ are determined.
In Section \ref{SectionE6zonotope}   the criterion for $P_v(\mathsf{E}_6) + Z(U)$ to be a parallelotope is given.
In Sections \ref{SectionEnumeration} and \ref{SectionOtherLattice} the enumeration of maximal feasible subsets is done for $\mathsf{E}_6$ and some other lattices.

\section{General results}\label{sec:gen_results}

\begin{figure}
\begin{center}
\begin{minipage}{5.7cm}
\centering
\epsfig{file=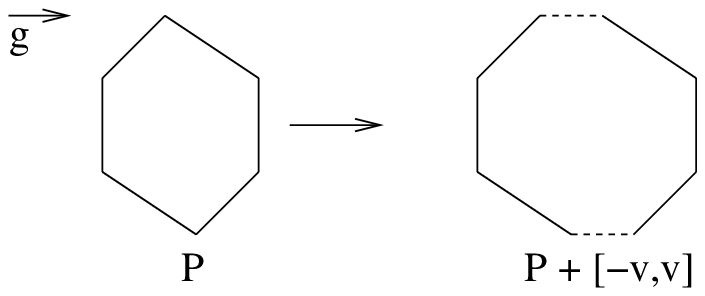,height=25mm}\par
(i) $v$ is not free
\end{minipage}
\begin{minipage}{5.7cm}
\centering
\epsfig{file=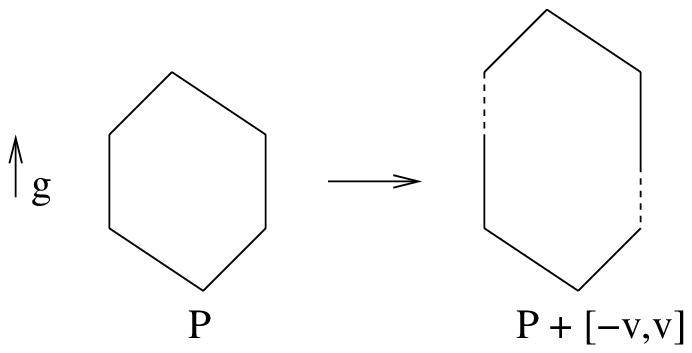,height=25mm}\par
(ii) $v$ is free
\end{minipage}
\end{center}
\caption{The image of a $6$-belt of $P$ in $P + [-v,v]$}
\label{SumOnSixBelt}
\end{figure}

Recall that a vector $x$ is called a facet vector of a parallelotope $P$ if it connects the center of $P$ with the center of a parallelotope
$P+x$ such that $P+x$ and $P$ share a facet.
If $P$ is a Voronoi parallelotope, then $x$ is orthogonal to the facet $P \cap (P+x)$.

In \cite{Do} Dolbilin introduced the notion of standard faces of parallelotopes. Namely, if $P$ and $P+x$ are parallelotopes of the face-to-face tiling
by translates of $P$, and if $F = P\cap (P+x)\neq \varnothing$,
then $F$ is called a {\it standard face} of $P$ and $x$ is called the {\it standard vector} of $F$. If $P$ is centered at the origin, then
$F$ has a center of symmetry at $x/2$. In what follows the definition of standard faces will be used extensively. 

The following Theorem gives a characterization of free vectors
that can be used for a computer enumeration of all free vectors for a given parallelotope $P$, or 
in a case by case analysis as shown below in our discussion of $\Esix$.

\begin{theor}(\cite{Gr})\label{bfr}
Let $P$ be a parallelotope and $v$ be a vector.
The following assertions are equivalent:
\begin{enumerate}
\item[(i)] $P$ is free along $v$ (i.e. $P+z(v)$ is a parallelotope);
\item[(ii)] $v^Tp=0$ for at least one facet normal $p$ of each $6$-belt of $P$; equivalently, $p$ is parallel to at least one facet of each 6-belt.
\end{enumerate}
\end{theor}

\proof If $v$ does not satisfy (ii) then Figure \ref{SumOnSixBelt}.(i)
shows that there is a $8$-belt in $P + z(v)$. So, (ii) is necessary.

Assume that (ii) is true. We start by showing that the first two Venkov conditions must then hold for $P+z(v)$. Since $P$ and $z(v)$ are centrally symmetric, 
the Minkowski sum $P+z(v)$ is centrally symmetric as well. If $F$ is a facet of $P + z(v)$ that is a translate of a facet of $P$, then it is necessarily centrally symmetric. If the facet $F$ is not such a translate, then it has the form $F = G + z(v)$, where $G$ is either a $(n-1)$ or $(n-2)$-face of $P$.
If $G$ is a facet of $P$, then it is centrally symmetric, and therefore the sum $G + z(v)$ is also symmetric. If $G$ is a $n-2$-face, 
then condition (ii) implies that $G$ belongs to a 4-belt of $P$ and is symmetric. For this reason $G + z(v)$ is also symmetric.

We now show that $P+z(v)$ has only $4$- and $6$-belts. We establish this result using a local argument --- we show that each $(n-2)$-face $G$ of 
the polytope $P+z(v)$ belongs to either 3 or 4 distinct translates of $P+z(v)$ that fit facet-to-facet around $G$. From this {\it local tiling property} it 
immediately follows that all belts of $P+z(v)$ have length 4 or 6. 

We need to consider the three possibilities for a $(n-2)$-face, $G_v$, of $P+z(v)$:
\begin{enumerate}
\item $G_v = G \pm v$ (this notation means that $F$ is a translate of $G$ by $v$ or $-v$), where $G$ is a $(n-2)$-face of $P$;
\item $G_v = G + z(v)$, where $G$ is a $(n-2)$-face of $P$;
\item $G_v = G + z(v)$, where $G$ is a $(n-3)$-face of $P$ (and the sum is direct);
\end{enumerate}
By condition (ii), the local tiling property for $G_v$ follows from that for $G$ in cases (1) and (2). So we need only consider case (3).

Consider the face-to-face tiling by parallel copies of $P$. Let $G$ be an $(n-3)$-face of the tiling and let $\pi_G$ be a projection along the linear space $\Lin(G)$ onto the complementary space $(\Lin(G))^{\perp}$. The projections of the tiles sharing $G$ split a sufficiently small neighborhood of $\pi_G(G)$ in the
same way as some $3$-dimensional complete polyhedral fan $\Fan (G)$ does. The {\it incidence type} of the face $G$ is a combinatorial type of $\Fan (G)$.

It is convenient to label the incidence type of a face (and therefore the combinatorial type of its fan) by its dual cell. 
By definition~\cite{Or}, the {\it dual cell} $D(G)$ of a face $G$ is the convex hull of centers of all parallelotopes in the tiling that have the face $G$.

The five possible combinatorial types of $\Fan (G)$, where $G$ is a $(n-3)$-face, were classified by B.N. Delaunay in \cite{De}. 
We label these types by their corresponding dual cells: (a) tetrahedron, (b) octahedron, (c) pyramid with quadrangular base, (d) prism with triangular base 
and (e) cube. In the cases (b) and (e) the face $G$ is standard. 

Since the face $G$ is $(n-3)$-dimensional, and $G+z(v)$ is an $(n-2)$-dimensional face of $P+z(v)$, the segment $\pi_G(z(v))$ is non-degenerate.
Further, if a $1$-dimensional face (ray) of $\Fan(G)$ is {\it trivalent}, i.e. incident to exactly 3 two-dimensional cones, then these two-dimensional cones
correspond to facets of the same $6$-belt. Thus for every trivalent ray of $\Fan(G)$ the segment $\pi_G(z(v))$ is parallel to at least one two-dimensional face
containing that ray.

We call the segment $\pi_G(z(v))$ {\it feasible} if for every trivalent ray of $\Fan(G)$ one of the two-dimensional cones incident to this ray is parallel to
$\pi_G(z(v))$. Respectively, we will call the face $G+z(v)$ a {\it feasible extension} of the face $G$.

We consider each of the five types of local structure for the $(n-3)$-faces $G$ of $P$, and determine how this structure can change when the individual tiles are replaced by their sum with a feasible $v$. Remarkably, in all cases these modified tiles fit back together as they initially did.

We enumerate all combinatorially distinct pairs of $(\Fan,z)$ with $\Fan$ a $3$-dimensional fan and $z$ a segment with the following properties.
\begin{enumerate}
\item  $\Fan$ is combinatorially equivalent to a fan of some $(n-3)$-face in some tiling of $\mathbb R^n$ by parallelohedra. 
\item  $z$ is feasible for $\Fan$.
\end{enumerate}
The way we do the enumeration is as follows: given a 3-dimensional fan $\Fan$ and a subset of its 2-dimensional faces, we can determine, if
the subset covers all trivalent rays of $\Fan$, and if there is a segment $z$ parallel to every 2-dimensional face from the chosen subset.
If the answer is positive for both instances, $z$ is feasible for $\Fan$. Of course, every possible pair $(\Fan, z)$ will be considered through
this enumeration.

As a result, we obtain a total of 13 types of pairs $(\Fan, z)$.
All these cases are listed and described explicitly in Figures~\ref{segm1} and~\ref{segm2}. In fact, each of the cases d.1), e.1) and e.2) covers an infinite 
series of segments for a fixed fan $\Fan$, however the treatment is the same within each single case.

\begin{figure}[!t]
\centerline{
\begin{minipage}{0.49\textwidth}
\centerline{\includegraphics[height=4.5cm]{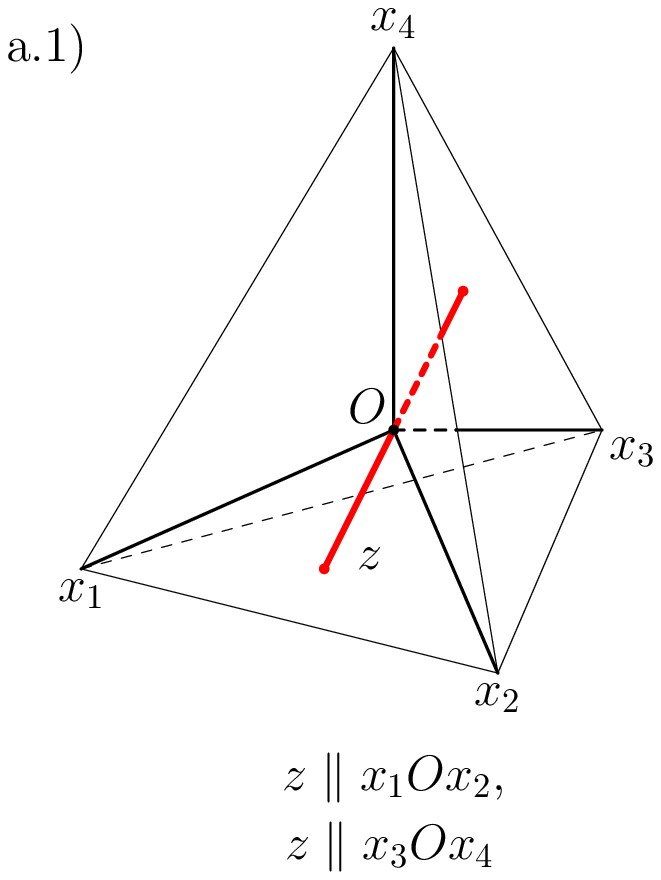}}
\end{minipage}
\begin{minipage}{0.49\textwidth}
\centerline{\includegraphics[height=4.5cm]{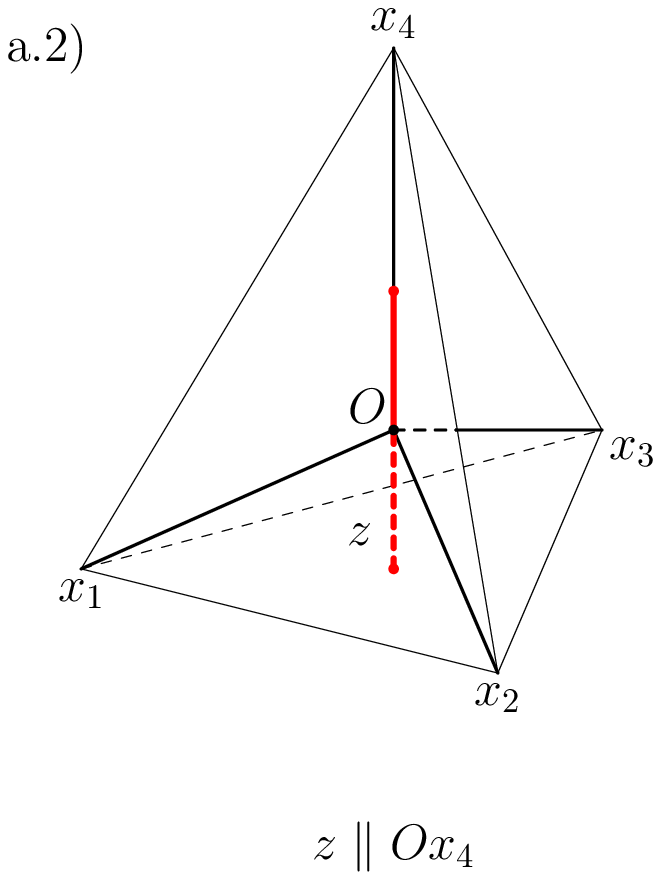}}
\end{minipage}
}

\vspace{5mm}
\centerline{
\begin{minipage}{0.49\textwidth}
\centerline{\includegraphics[height=4.5cm]{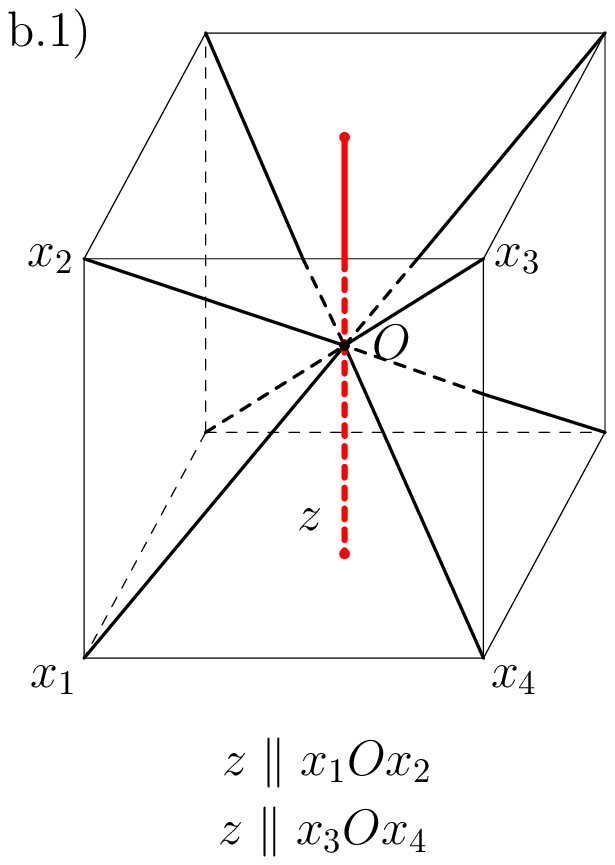}}
\end{minipage}
\begin{minipage}{0.49\textwidth}
\centerline{\includegraphics[height=4.5cm]{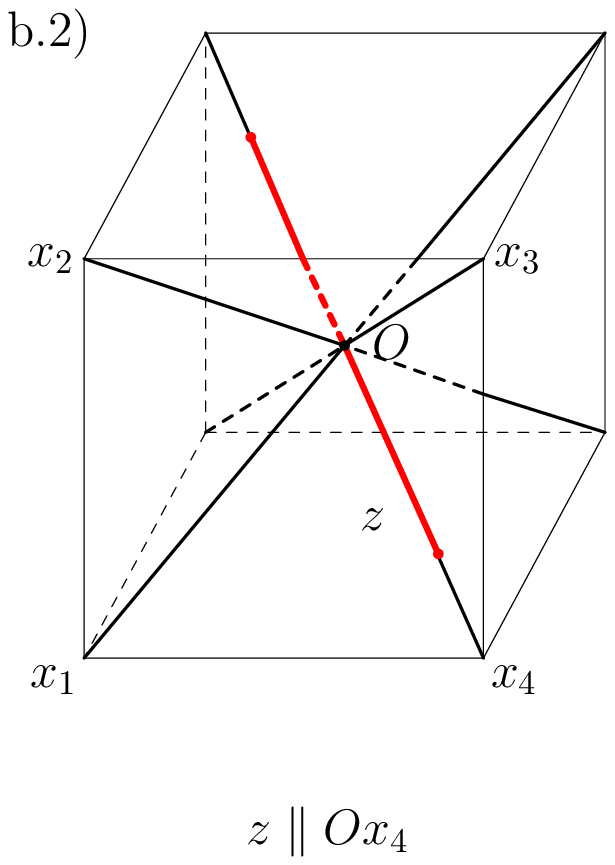}}
\end{minipage}
}

\vspace{5mm}
\centerline{
\begin{minipage}{0.32\textwidth}
\centerline{\includegraphics[height=4.5cm]{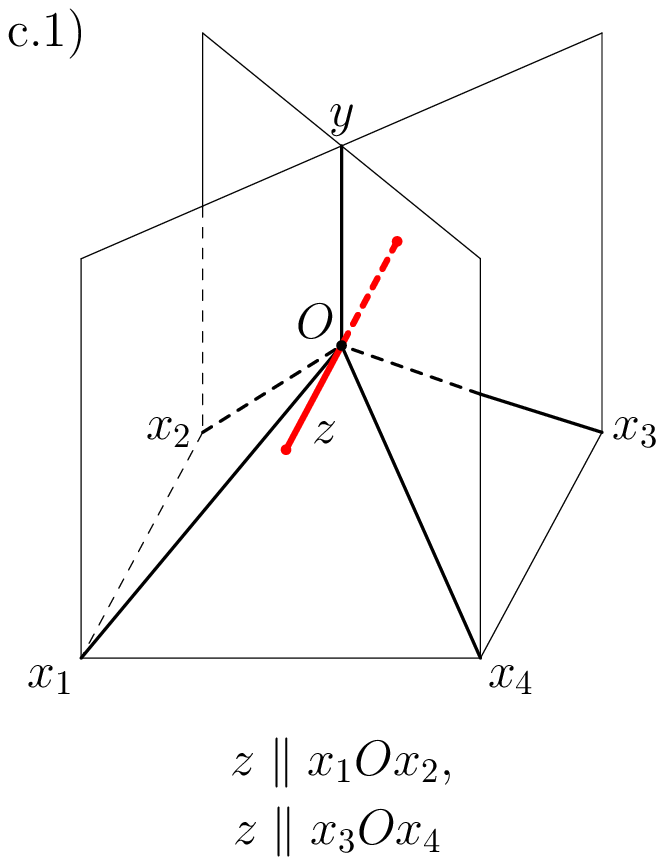}}
\end{minipage} \qquad
\begin{minipage}{0.32\textwidth}
\centerline{\includegraphics[height=4.5cm]{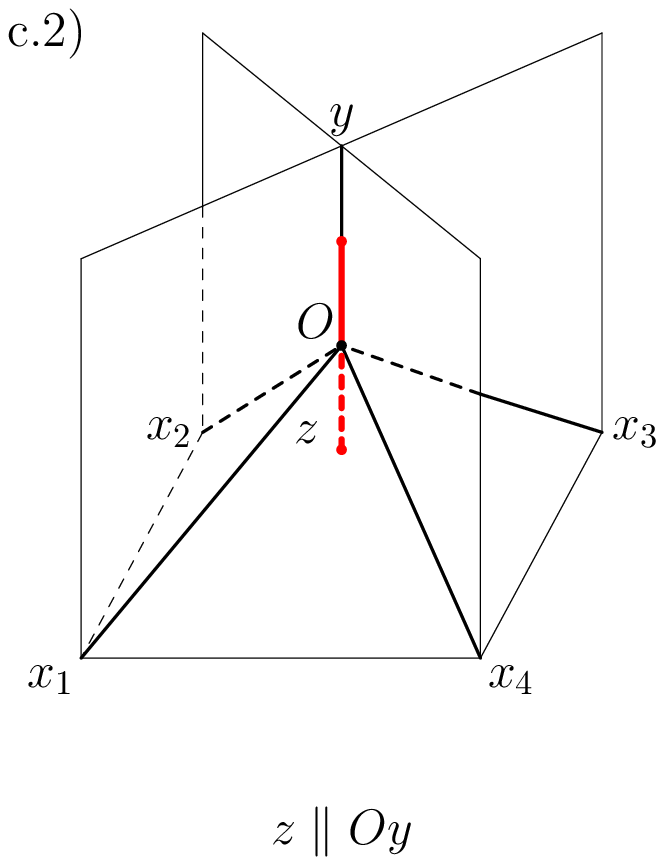}}
\end{minipage}
\begin{minipage}{0.32\textwidth}
\centerline{\includegraphics[height=4.5cm]{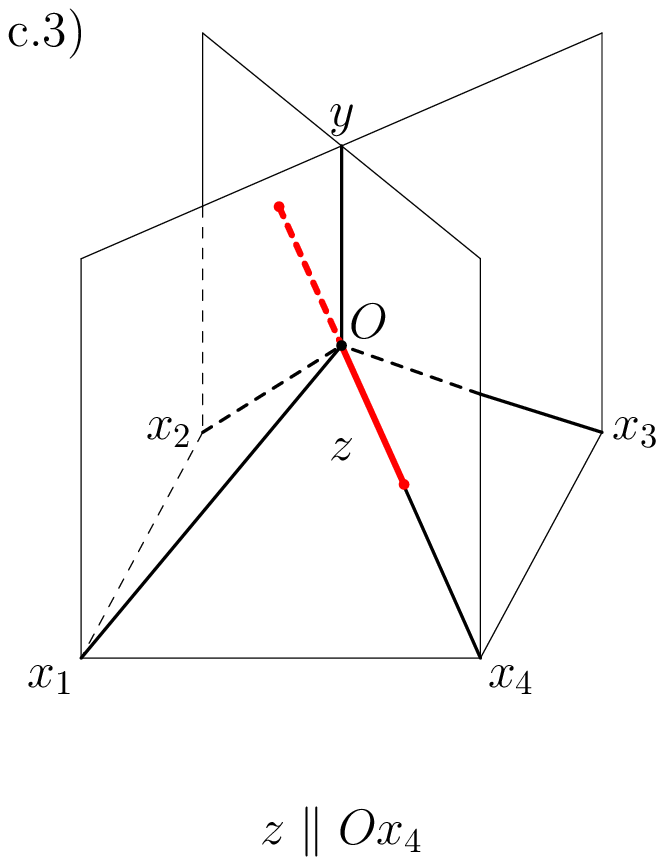}}
\end{minipage}
}

\vspace{3mm}
\caption{Possible arrangements of free segments and $(n-3)$-faces}
\label{segm1}
\end{figure}

\begin{figure}[!t]
\centerline{
\begin{minipage}{0.32\textwidth}
\centerline{\includegraphics[height=4.5cm]{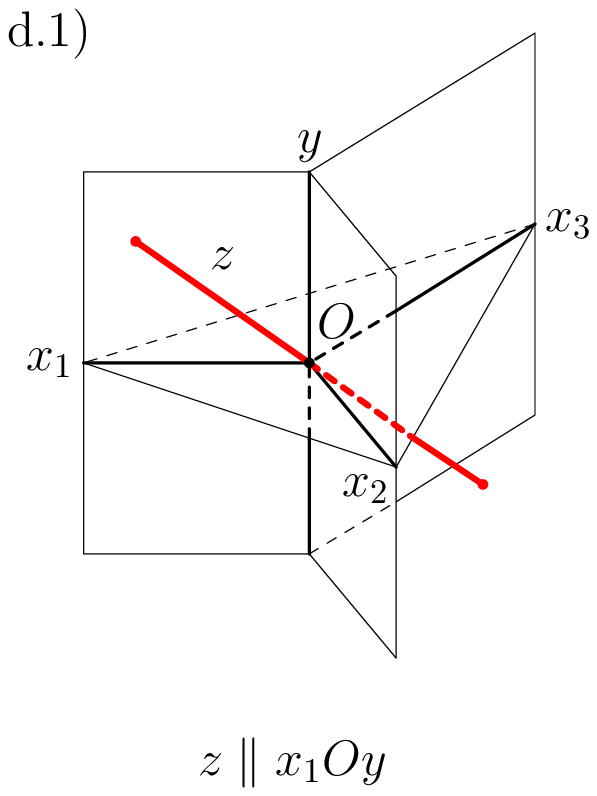}}
\end{minipage} \qquad
\begin{minipage}{0.32\textwidth}
\centerline{\includegraphics[height=4.5cm]{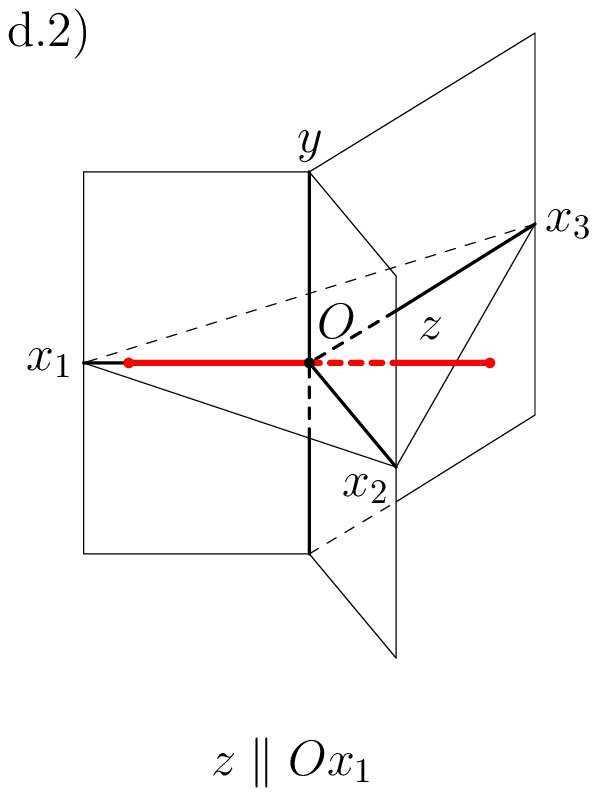}}
\end{minipage}
\begin{minipage}{0.32\textwidth}
\centerline{\includegraphics[height=4.5cm]{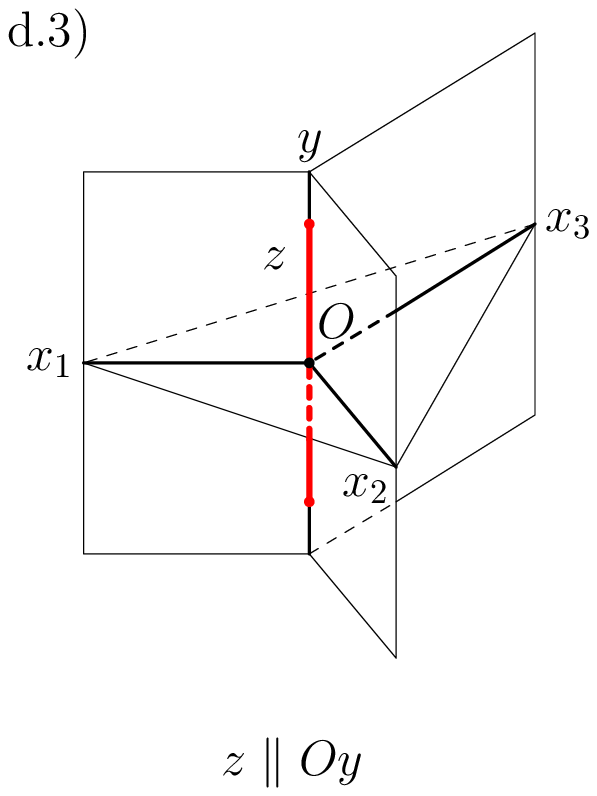}}
\end{minipage}
}

\vspace{5mm}
\centerline{
\begin{minipage}{0.32\textwidth}
\centerline{\includegraphics[height=4.5cm]{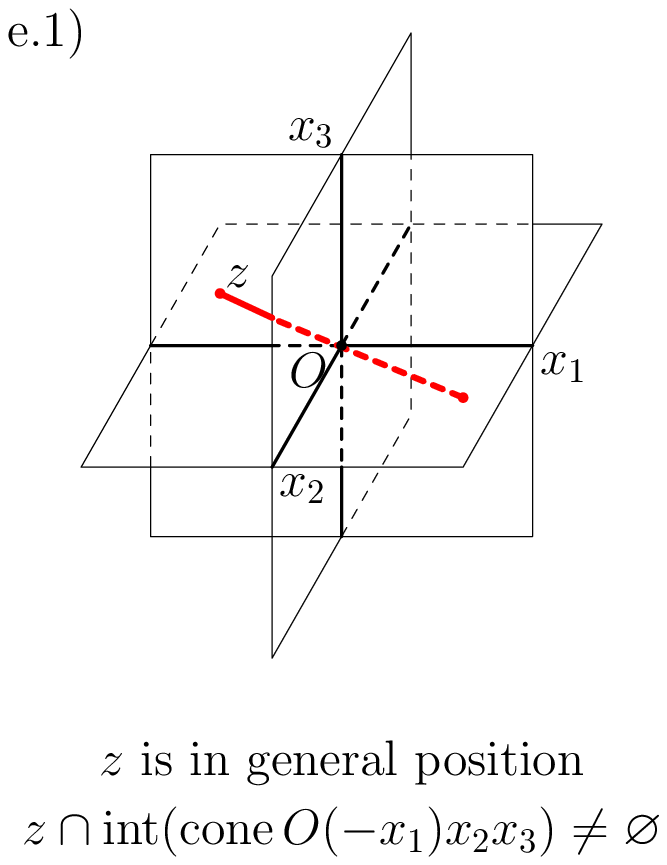}}
\end{minipage} \qquad
\begin{minipage}{0.32\textwidth}
\centerline{\includegraphics[height=4.5cm]{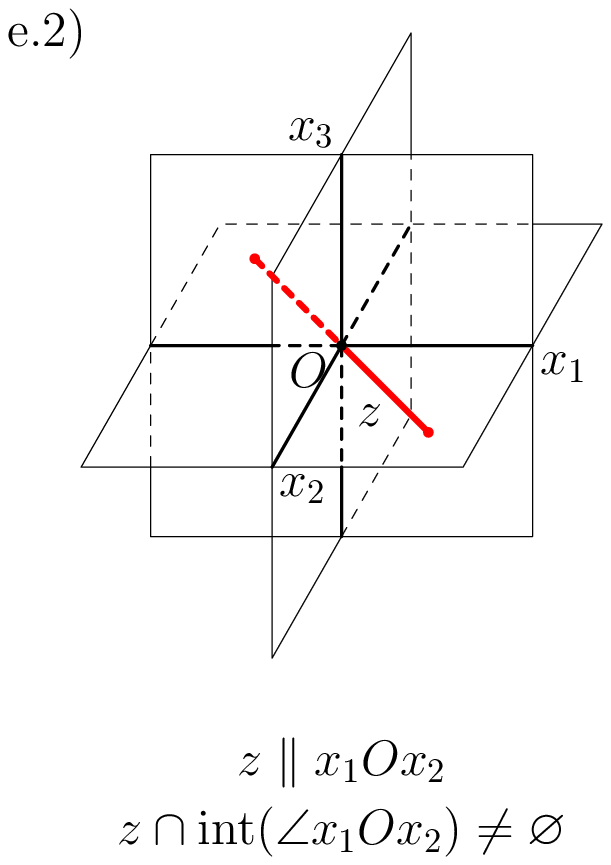}}
\end{minipage}
\begin{minipage}{0.32\textwidth}
\centerline{\includegraphics[height=4.5cm]{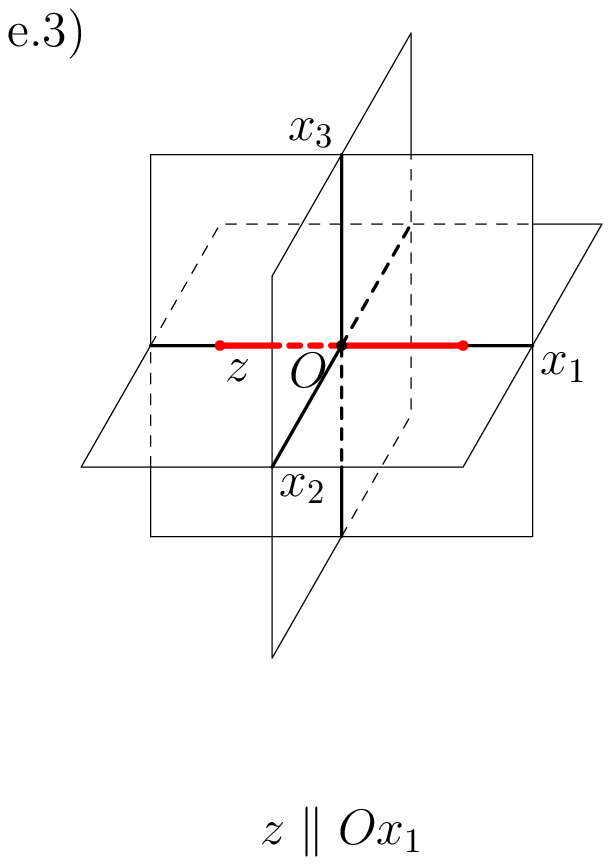}}
\end{minipage}
\vspace{3mm}
}
\caption{Possible arrangements of free segments and $(n-3)$-faces, continued}
\label{segm2}
\end{figure}

We also emphasize that each of the 13 cases has its own subcases, depending on which 3-dimensional cone of $\Fan$ corresponds to $P$. Still in each case
it is easy enough to treat all subcases simultaneously.

In each subcase one can prove that either there is no feasible extension of the face $G$, or the copies of $P+z(v)$ fit together around $G+z(v)$. The details
are provided in Appendix A.  

Therefore the third Venkov condition for $P+z(v)$ is fulfilled. Thus $P+z(v)$ is a parallelotope. \qed

In \cite{GrSegment} a graph $\Gamma(P)$ is built as follows for every parallelotope $P$. The vertices of $\Gamma(P)$ represent facets and
two vertices are said to be adjacent if there exists a $6$-belt containing the corresponding facets. If one contracts all pairs of vertices
of $\Gamma(P)$ corresponding to pairs of antipodal facets, the result will be the so called {\it red Venkov graph} $\Gamma_{RV}(P)$.
$\Gamma_{RV}(P)$ is connected iff $\Gamma(P)$ is connected.

$P$ is a direct sum of two non-trivial parallelotopes $P_1$ and $P_2$ if and only if the
red Venkov graph of $P$ is not connected \cite[Proposition 4]{GrSegment} (see also \cite{Or,Ordine2,rybnikov}).
The following is an extension of Theorem \ref{bfr} to Voronoi polytopes:

\begin{theor}\label{EquivForVoronoi}
Let $P$ be a non-decomposable Voronoi parallelotope and $v$ be a vector.
The following assertions are equivalent:
\begin{enumerate}
\item[(i)] $P + z(v)$ is affinely equivalent to a Voronoi polytope.
\item[(ii)] There exists $a>0$ such that for any $6$-belt with facet vectors $\{p_1, \dots, p_6\}$
either $v^T p_i = 0$ for $1\leq i\leq 6$ or $v^T p_i = 0$ for two indices and $v^T p_i = \pm a$ for the other four indices.
\end{enumerate}
\end{theor}
\proof See \cite[Lemma 3]{GrSegment}. \qed

\begin{cor}
Let $P$ be a non-decomposable parallelotope and assume that Voronoi conjecture is true. Then $P$ is finitely free.
\end{cor}
\proof Since $P$ is non-decomposable, every facet vector $w$ belongs to at least one $6$-belt.
So, if $v$ is a free vector then we have $w^T v\in \{-a, 0, a\}$.
Since the facet vectors span $\RR^n$ the linear system has a unique solution
for a given choice of signs and zeros. The result follows by remarking that
there is only a finite number of possible choices of zeros of signs. \qed

The connection between Voronoi conjecture and finite freedom is noticed in \cite{Vegh}.

\section{Lattice and standard vectors of $P+z(v)$}\label{sec:standard_vector}

Let $\Lambda(P)$ be the lattice generated by facet vectors of a parallelotope $P$, i.e. the lattice such that
$$\{P+x : x \in \Lambda(P) \} $$
is a face-to-face tiling. Let $v$ be a non-zero free vector of $P$. In this section we will prove several relations
between the lattices $\Lambda(P+z(v))$ and $\Lambda(P)$, as well as between the sets of standard vectors of $P$ and $P+z(v)$.

\begin{lem}\label{lem:operator}
$\Lambda(P+z(v)) = A_v \Lambda(P)$, where $A_v x = x + 2 n_v(x) v$ and $n_v(x)=e_v^T x$ is a linear function from $\Lambda(P)$ to $\ZZ$.
\end{lem}

\proof Find all facet vectors of $P+z(v)$. As mentioned before, each facet $F$ of $P+z(v)$ is of one of the following types.
\begin{itemize}
\item $F = F'+z(v)$, where $F'$ is a facet of $P$, $F' \parallel z(v)$. Then if $s$ and $s'$ are the facet vectors of $F$ and $F'$ respectively, then $s = s'$.
\item $F = F'+v$, where $F'$ is a facet of $P$. Then $s = s' + 2v$.
\item $F = F'-v$, where $F'$ is a facet of $P$. Then $s = s' - 2v$.
\item $F = F'+z(v)$, where $F'$ is a standard $(n-2)$-face of $P$ and $v$ is strongly transversal to $F'$. If $s'$ now denotes the standard vector of $F'$, then $s = s'$.
\end{itemize}

From here it follows that every point of $\Lambda(P+z(v))$ is $x + 2n v$, where $x\in \Lambda(P)$ and $n\in \mathbb Z$. Moreover, for every 
$x\in \Lambda(P)$ there exists at least one point of $\Lambda(P+z(v))$ of that form.

Now suppose that $x + 2n_1 v$ and $x + 2n_2 v$ are points of $\Lambda(P+z(v))$. Then, obviously, 
$$x + 2\lambda n_1 v \in \Lambda(P+\lambda z(v)) \quad \text{and} \quad x + 2\lambda n_2 v \in \Lambda(P+\lambda z(v))$$
for every $\lambda > 0$. Then there are two points of $\Lambda(P+\lambda z(v))$ which are arbitrarily close to each other as
$\lambda \to 0$. This is impossible, so $n$ is a function of $x$. 

Since we suppose that $P+z(v)$ is a parallelotope, it is of non-zero width along $v$.
Venkov asserts and proves \cite[item (2) of Theorem 1]{venkovProjection} that facet vectors of all facets parallel to $v$ generate an $(n-1)$-dimensional sublattice if a parallelotope has non-zero width in direction $v$.
In our case, denote this sublattice of $\Lambda(P+z(v))$ by $L_v$. Obviously, we have also $L_v\subset \Lambda(P)$. The lattice $L_v$ determines a partition of $\Lambda(P)$ into $(n-1)$-dimensional layers.

The layers are equally spaced and cover the entire lattice $\Lambda(P)$. Thus there exists a vector $e_v$ such that the scalar product $e_v^T x$
runs through all integers while $x$ runs through $\Lambda(P)$. 

Put $n_v(x) = e_v^T x$ and $A_v x = x + 2 n_v(x) v$ as in the condition of the Lemma. Then every facet vector of the parallelotope $P+z(v)$ has the form of
$A_v s$, where $s$ is some standard vector of $P$, and every facet vector of the parallelotope $P$ has the form of $A_v^{-1} s'$, where $s$ is a facet vector of 
$P+z(v)$. Hence $A_v$ is an isomorphism between $\Lambda(P)$ and $\Lambda(P+z(v))$. \qed

\begin{rem}
It is not hard to see that $n_v(x) = n_{\lambda v}(x)$ for every $\lambda>0$.
\end{rem}

\begin{lem}\label{lem:depend_hered}
Let $x_1, x_2, \ldots, x_k$ be points of $\Lambda(P)$. Then the points
$$A_v x_1, A_v x_2, \dots, A_v x_k$$
are affinely dependent if and only if $x_1, x_2, \ldots, x_k$ are affinely dependent.
\end{lem}
\proof The matrix $A_v$ is non-singular since $\Lambda(P+z(v))$ is full-dimensional. Hence the result follows. \qed

\begin{lem}\label{lem:st_vec_hered}
Let $x\in \Lambda(P)$. Then the following conditions are equivalent.
\begin{enumerate}
\item[(i)] $A_v x$ is a standard vector for $P+z(v)$.
\item[(ii)] $x$ is a standard vector for $P$ and $n_x(v)\in \{0, \pm 1 \}$.
\end{enumerate}
\end{lem}

\proof In this proof we write $P+x$ for the translate of $P$ centered at $x$. This coincides with the standard notation under assumption that $P$ is
centered at the origin.

(ii) $\Rightarrow$ (i). If $x$ is a standard vector, then $P+x$ and $P$ share the point $x/2$. Under assumption $n_v(x) \in \{0, \pm 1 \}$
it immediately follows that 
$$(P + z(v)) + A_v x \quad \text{and} \quad P+z(v)$$ 
share the point  $A_v x/2$, which is equivalent to (i).

(i) $\Rightarrow$ (ii). Suppose that the parallelotopes 
$$(P + z(v)) + A_v x \quad \text{and} \quad (P+z(v))$$
share the point  $z = A_v x/2$.
Let $\ell$ be the line passing through $z$ and parallel to $v$.
Since the tiling is face-to-face, then there are three options.
\begin{itemize}
\item For some $\alpha \geq 1$ we have
\begin{equation} \label{eq:stand1}
\left\{(P + z(v)) + A_v x \right\}\cap \ell = 
(P + z(v)) \cap \ell = [v(z-\alpha v), v(z+\alpha v)]
\end{equation}

\item For some $\alpha \geq 1$ we have
\begin{equation} \label{eq:stand2}
\left\{\begin{array}{rcl}
\left\{(P + z(v)) + A_v x \right\}\cap \ell &=& [z-2\alpha v, z]\\
(P + z(v)) \cap \ell &=& [z, z+2\alpha v]
\end{array}\right.
\end{equation}

\item For some $\alpha \geq 1$ we have
\begin{equation} \label{eq:stand3}
\left\{\begin{array}{rcl}
\left\{(P + z(v)) + A_v x \right\} \cap \ell &=& [z, z+2\alpha v]\\
(P + z(v))  \cap \ell &=& [z-2\alpha v, z]
\end{array}\right.
\end{equation}
\end{itemize}
 
Now notice that for every $\lambda>0$ there is a face-to-face tiling by translates of $P+\lambda z(v)$. Using this, consider all three cases separately.

If Equation \eqref{eq:stand1} holds, then $n_v(x) = 0$. Otherwise the tiling by $P+\lambda z(v)$ is not face-to face for $\lambda = 1+\varepsilon$
(and for $\lambda = 1-\varepsilon$) if $\varepsilon$ is a sufficiently small positive number. Indeed, one can check that 
$$(P + \lambda z(v)) + A_{\lambda v} x \cap (P + \lambda z(v))$$
is nonempty, but this intersection is not a face of any of the two parallelotopes.

If Equation \eqref{eq:stand2} holds, then $n_v(x) = -1$. Indeed, if $n_v(x) < -1$, then the tiling by $P+\lambda z(v)$ is not face-to face for 
$\lambda = 1-\varepsilon$, where $\varepsilon$ is a sufficiently small positive number. If $n_v(x) > -1$, then the tiling by $P+\lambda z(v)$ is 
not face-to face for $\lambda = 1+\varepsilon$.

The case of Equation \eqref{eq:stand3} does not differ with the case of \eqref{eq:stand2}, and we get $n_v(x)=1$. \qed

\section{Algorithms for computing free vectors}\label{SectionEnumAlgo}

For highly symmetric lattices, some efficient techniques for computing
their Delaunay tessellation and so by duality Voronoi polytopes have been
introduced in \cite{complexity}.
These techniques use the quadratic form viewpoint for the actual computation
and can be adapted to the enumeration of free vectors and strongly regular
faces.

For the enumeration of free vectors, Theorem \ref{bfr} gives implicitly
a method for enumerating them. The first step is, given a lattice $L$, to use \cite{complexity}
in order to get the Delaunay tessellation $\mathcal D(L)$. The triangular faces of
$\mathcal D(L)$ enumerate all 6-belts of the parallelotope $P_V(L)$.
By Theorem \ref{bfr}, every free vector of $P_V(L)$ must satisfy an orthogonality condition 
for each $6$-belt. So, if we have
$N$ $6$-belts, then we have $3^N$ cases to consider, which can be large.
The enumeration technique consider the $6$-belts one by one by making
choice at each step. We use symmetries to only keep non-isomorphic
representatives of all choices.
We also use the fact that any choice among the three possibilities implies
a linear equality on the coefficients of the free vector $v$.
Hence the choice made for some $6$-belts might imply other choices for
other $6$-belts. So, the dimension decreases at each step and the number
of choices is thus only $3^m$ at most, where $m=\min(n,N)$ and $n$ is dimension of $L$.
At the end we have a number of vector spaces containing the free vectors.
$P_V(L)$ is nonfree, respectively finitely free, if and only if all the vector
spaces are $0$-dimensional, respectively at most $1$-dimensional.

The enumeration of strongly transversal faces can be done in the following
way: If $G$ is a strongly transversal face of $P$, then any subface of $G$
is also strongly transversal. Thus starting from the vertices of $P$, which
correspond to Delaunay polytopes of $P$, we can enumerate all strongly
transversal faces of $P$ and hence describe the facet and belt structure
of $P + Z(U)$ with $U$ a set of free vectors. By using Venkov's condition
this allows to determine whether or not $P + Z(U)$ is a parallelotope.

Another variant is to write $U=\{u_1, \dots, u_p\}$ and write
\begin{equation*}
P + Z(U) = P' + z(u_p) \mbox{~with~} P'=P + Z(\{u_1, \dots, u_{p-1}\}).
\end{equation*}
If we know that $P'$ is a Voronoi polytope, then we can test whether
$P + Z(U)$ is a parallelotope by testing whether $u_p$ is a free vector
via Theorem \ref{bfr}.
Furthermore, by using the sign condition from Theorem \ref{EquivForVoronoi}
we can test whether $P + Z(U)$ is a Voronoi polytope.
The method relies on computing the Delaunay tessellation.

We choose to use the second method because it allows us to distinguish
between parallelotope and Voronoi polytope.
The process of enumeration is then done by considering all subsets $U$
of the set of free vectors and adding vectors one by one and testing whether
they are feasible or not.
By ${\mathcal F}_{min}(L)$ we denote the set of minimal forbidden subsets
of $L$.
Similarly ${\mathcal F}_{max}(L)$ denotes the maximal feasible subsets
of $L$.
Key information on those subsets are given in Table \ref{FreeInformationForLattices} for $12$ lattices.
The cost of computing the Delaunay tessellation is relatively expensive,
hence we always use the list of already known forbidden subsets in order
to avoid such computation whenever possible.
We found out that in the case that we consider in Section
\ref{SectionOtherLattice}, whenever a polytope $P_V(L) + Z(U)$ is a
parallelotope then it is also a Voronoi polytope, thereby confirming
Voronoi's conjecture in those cases.

However, for $\Esix$ we prefer to use the first enumeration method. Then we confirm the outcome
by an explicit proof.

\section{The root lattices $\Esix$ and $\Esix^*$}\label{SectionRootlattice}

There are exactly $27$ straight lines on any smooth non-degenerate cubic
surface in the $3$-dimensional projective space. This fact is very well known
(see, for example, \cite{Cox}, and other papers of Coxeter). It
is proved in many textbooks on algebraic geometry (see, for example,
\cite[ch.3,\S 7]{Re}, \cite[ch.IV, \S 2.5]{Sha}). The combinatorial
configuration of the set $\mathcal L$ of these $27$ lines is unique. For
example, if two lines $l,l'\in\mathcal L$ intersect, then there is a unique
line $l''\in\mathcal L$ that intersects both lines $l,l'$. Every three mutually
intersecting lines generate a tangent plane of the cubic surface.
Each line intersects exactly $10$ lines and belongs to exactly $5$ of
all $45$ tangent planes.

Schl\"afli described his famous {\em double six}
\[\begin{array}{cccccc}
a_1 &a_2& a_3& a_4& a_5& a_6\\
b_1& b_2& b_3& b_4& b_5& b_6
\end{array}\]
which is a special arrangement of twelve lines from $\mathcal L$ on the surface
such that any two of them intersect if and only if they occur in different rows
and different columns. Any two columns determine a pair of planes $a_ib_j$
and $a_jb_i$ whose intersection $c_{ij}=c_{ji}$ intersects all the four
lines and therefore must lie entirely on the surface. In this way Schl\"afli
obtained his notations $a_i,b_j, c_{ij}$ for all $27$ lines.

Burnside \cite[pp.485--488]{Bur}, used these symbols of $27$ lines as
elements of an algebra (in fact as vectors in $6$-dimensional space). This
amounts to representing the lines by $27$ points in an affine $6$-space, such
that the $45$ triangles representing the tangent planes all have the same
centroid. Using this centroid as the origin, he applies Schl\"afli's symbols
for the $27$ lines to the positions of vectors of the $27$ points, so that
\begin{equation}\label{abc}
a_i+b_j+c_{ij}=0, \mbox{  }c_{ij}+c_{kl}+c_{mn}=0, \mbox{ where }
\{ijklmn\}=\{123456\}=I_6.
\end{equation}

Choose a set $\mathcal A$ of six vectors $a_1, a_2 \ldots, a_6 \in \mathbb R^6$ so that
$a_i^T a_i = 4/3$ and $a_i^T a_j = 1/3$ for all possible $i \neq j$. One can see that such set $\mathcal A$ exists and forms a basis of $\mathbb R^6$.

Define
$$h=\frac{1}{3}\sum_{i\in I_6}a_i, $$
and let $b_i=a_i-h$, and $c_{ij}=h-a_i-a_j$ for all possible $i \neq j$.

The lattice integrally generated by the basis $\mathcal A$ and the vector
$h$ is $\Esix^*$ which is dual to the root lattice $\Esix$.
This representation of the lattice $\Esix^*$ was used by Baranovskii
in \cite{Ba} for a description of Delaunay polytopes of $\Esix^*$.
Barnes \cite[Formula (8.5)]{Brn} uses the basis $\mathcal A$ to
describe minimal vectors of $\Esix^*$. Vectors of $\Esix^*$
have all coordinates in the basis $\mathcal A$ equal to one third of an
integer.
The $27$ vectors of the set
\begin{equation}\label{min}
{\mathcal M}=\left\{a_i,b_i=a_i-h, c_{ij}=h-a_i-a_j, \mbox{ where }
1\le i<j\le 6\right\},
\end{equation}
are, up to sign, the $27$ minimal vectors of the lattice
$\Esix^*$.
The lattice $\Esix^*$ has an automorphism group equal to
$W(\Esix)\times \{\pm Id_6\}$ with $W(\Esix)$ being the Weyl group
of $\Esix$ (see \cite{humphreyscoxeter} for more details),
which is also the automorphism group of ${\mathcal M}$.

Let $A$ be the Gram matrix for the basis $\mathcal A$. Then the inverse matrix $E = A^{-1}$, which is the Gram matrix of the dual basis
${\mathcal E}=\{e_1, e_2, \ldots, e_6\}$, has elements
\begin{equation*}
e_{ii}=e_i^2=\frac{8}{9}, \mbox{  }i\in I_6, \mbox{  }
e_{ij}=e_{ji}=e_i^Te_j=-\frac{1}{9}, \mbox{  }i\neq j.
\end{equation*}
It is easy to check that vectors of the dual lattice
$\Esix$ have in the basis $\mathcal E$ the following form
\[\sum_{i\in I_6}z_ie_i, \mbox{ where } z_i\in\ZZ \mbox{ and }
\sum_{i\in I_6}z_i\equiv 0 \pmod 3. \]

\section{The Schl\"afli polytope $P_{Schl}$}\label{SectionSchlafli}
The convex hull of end-points of all vectors of the set $\mathcal M$ is the
Schl\"afli polytope $P_{Schl}$, i.e., $P_{Schl}=\conv\mathcal M$.
Since $P_{Schl}$ is a Delaunay polytope of the lattice $\Esix$ and
the Voronoi polytope $P_V(\Esix)$ of the lattice $\Esix$ is
the convex hull of $P_{Schl}$ and its centrally symmetric copy $P^*_{Schl}$
(see \cite{CS}), we have to study properties of $P_{Schl}$.

Two vertices $q, q' \in \mathcal M$ of $P_{Schl}$ are adjacent by an edge if and only if $q^Tq'=\frac{1}{3}$, i.e., if the corresponding lines of
$\mathcal L$ are skew. Let $X\subset\mathcal M$ be a subset of cardinality 6 such that $q^Tq'=\frac{1}{3}$ for all $q,q'\in X, q\not=q'$. Then 
$\conv X$ is a simplicial facet $\alpha_5(r)$, where $r = 1/3 \sum\limits_{q\in X} q$.

Let $q\in \mathcal M$. Then there are 5 non-ordered pairs 
$$q'_i, q''_i \in \mathcal M, \quad i = 1, 2, \ldots, 5 \quad \text{with }\; q + q'_i + q''_i = 0. $$ 
Denote by $T(q)$ the set $\{q'_1, q''_1, q'_2, q''_2, \ldots, q'_5, q''_5 \}$. Then $\beta_5(q) = \conv T(q)$
is a cross-polytopal facet of $P_{Schl}$.

The polytope $P_{Schl}$ has two types of facets: $36$ pairs of simplicial
facets $\alpha_5(r)$ and $\alpha_5(-r)$; and $27$ cross-polytopal facets of form $\beta_5(q)$. Each facet $\beta_5(q)$ is orthogonal to the 
corresponding vector $q\in\mathcal M$ and is opposite to the vertex $q$. Both types of facets are regular polytopes with facets $\alpha_4$. Thus
all faces of $P_{Schl}$ of dimension $k\le 4$ are regular simplices $\alpha_k$.

Let $I_5=\{1,2,\dots,5\}$, and let $q'_i$, $q''_i$, $i\in I_5$, be the five pairs of opposite vertices of the facet $\beta_5(q)$. 
For every $J\subseteq I_5$ and $J'=I_5 \setminus J$ define
$$T_J(q)=\left\{q'_i : i\in J\}\cup\{q''_i : i\in J'\right\}.$$
Obviously, for each of the 32 subsets $J\subseteq I_5$ the polytope $\conv T_J(q)$ is a $4$-simplex which is a facet of $\beta_5(q)$, and vice versa:
if $\alpha_4$ is a facet of $\beta_5(q)$, then $\alpha_4 = \conv T_J(q)$ for some $J\subseteq I_5$.

\begin{table}\label{FirstTable}
\begin{center}
\begin{tabular}{|r||c|c|c|c|c|c|c|}
\hline
$\dim G$ &$0$ & $1$ & $2$& $3$& $4$& $5$& $6$\\
\hline
type of $G$ & vertex $\alpha_0$ & edge $\alpha_1$ & $\alpha_2$ &$\alpha_3$ &$\alpha_4$   &$\alpha_5$ or $\beta_5$ &$P_{Schl}$\\ \hline
$n(G)$      &$27$               & $216$           & $720$      & $1080$    & $432 + 216$& $72 + 27$               &$1$      \\ \hline
\end{tabular}
\end{center}
\caption{Faces $G$ of the Schl\"afli polytope $P_{Schl}$.}
\end{table}

\begin{lem}
\label{TJ}
Any subset $X\subseteq{\mathcal M}$ of cardinality $|X|=5$ such that
$p^Tp'=\frac{1}{3}$ for all $p,p'\in X$, $p\not=p'$, is a set $T_J(q)$ for
some $q\in{\mathcal M}$ and $J\subseteq I_5$.
\end{lem}
\proof It is not hard to see that $\conv X = \alpha_4(X)$ is a
$4$-face of $P_{Schl}$. 
Therefore, $\alpha_4(X)$ is a $4$-face either of a facet
$\beta_5(q)$ for some $q\in{\mathcal M}$, or of a facet $\alpha_5(r)$ for
some $r\in{\mathcal R}$. But each facet of type $\alpha_5(r)$ is contiguous
in $P_{Schl}$ by a $4$-face only to facets of type $\beta_5(q)$. This means
that each $4$-face of $P_{Schl}$ is a $4$-face of some $\beta_5(q)$. 
The vertex set of each $4$-dimensional subface of $\beta_5(q)$ equals $T_J(q)$ for some $J\subseteq I_5$. \qed

\section{The Voronoi polytope $P_V(\Esix)$}\label{SectionVoronoi}

The convex hull of $P_{Schl} = \conv \mathcal M$
and its centrally symmetric copy $P^*_{Schl} = \conv (- \mathcal M)$ is the Voronoi polytope of the
lattice $\Esix$, i.e.,
\[P_V(\Esix)=\conv(P_{Schl}\cup P^*_{Schl})=\diplo(P_{Schl}), \]
(see \cite{CS}, where the notation $\conv(P\cup P^*)=\diplo(P)$ is
introduced). We also mention the notable fact that the intersection $P_{Schl}\cap P^*_{Schl}$
is the Voronoi polytope of the dual lattice $\Esix^*$.

$P_V(\Esix)$ has $27$ pairs of opposite vertices $q, -q$ 
for $q\in\mathcal M$, and $36$ pairs of parallel opposite facets 
$$F(r)=\conv\bigl(\alpha_5(r)\cup (-\alpha_5(-r))\bigr)=\diplo(\alpha_5).$$

The edges of $P_V(\Esix)$ are of two types: $r$- and
$m$-edges. Vertices $q$ and $q'$ are adjacent by an $r$-edge if and
only if $q^Tq'=\frac{1}{3}$, and either $q,q'\in\mathcal M$ or
$q,q'\in-\mathcal M$. Each $m$-edge connects a vertex $q \in \mathcal M$ to a
vertex $-q' \in -\mathcal M$.
The $m$-edge between $q$ and $q'$ exists if and only if $q^Tq'=-\frac{2}{3}$.

Dolbilin in \cite{Do} shows that there is a one-to-one correspondence between minimal vectors of a coset of $L/2L$ and standard faces of the Voronoi polytope $P_V(L)$. A face $G$ of $P_V(L)$ is called {\em standard} if $G=G(t)=P_V(L)\cap P(t)$, where $P(t)=P_V(L)+t$ and $t$ is a minimal vector of a coset $L/2L$. If $G(t)$ is a facet of $P_V(L)$, then $t$ coincides with its facet vector.

Besides, there is a one-to-one correspondence between standard faces of $P_V(L)$ and centrally symmetric faces of the Delaunay tiling. Let $D(G)$ be a centrally symmetric Delaunay face related to a standard face $G(t)$. Then $t$ is a diagonal of the centrally symmetric Delaunay face $D(G)$. All other diagonals of $D(G)$ are all other minimal vectors of the coset containing $t$.

Translates and centrosymmetrical copies of $P_{Schl}$ form the Delaunay tiling for $\Esix$. Hence all minimal vectors of $\Esix$ 
are exactly all vectors connecting pairs of vertices of $P_{Schl}$. These vectors naturally split into two classes:
$$\mathcal R = \{ p-p' : p, p' \in \mathcal M, p^Tp' = 1/3 \} \quad \text{and}$$
$$\mathcal T = \{ p-p' : p, p' \in \mathcal M, p^Tp' = -2/3 \}.$$

If $p^Tp' = 1/3$, then $[v(p), v(p')]$ is an edge of $P_{Schl}$, i.e. a $1$-dimensional Delaunay cell of $\Esix$. Thus $\mathcal R$ is the set of facet vectors of
$P_V(\Esix)$. The notation $F(r)$ for the facet of $P_V(\Esix)$ with $r$ being its facet vector coincides with the definition of $F(r)$ above.

Two facets $F(r)$ and $F(r')$ of $P_V(E_6)$ intersect by a $4$-face if and only if $r^Tr'=1$. In this case $r''=r-r' \in \mathcal R$, and
$r^Tr''=-(r')^Tr''=1$. Hence the six facets $F(\pm r),F(\pm r'),F(\pm r'')$ form a $6$-belt. The Voronoi polytope $P_V(E_6)$ has 120 $6$-belts
studied in Section~\ref{SectionFreedom}.

If $p^Tp' = -2/3$, then $[v(p),v(p')]$ is a diagonal of a facet of $P_{Schl}$. Consequently, each  vector $t\in\mathcal T$ is a 
standard vector of some $m$-edge of $P_V(\Esix)$. Denote the endpoints of that edge by $a(t)$ and $-b(t)$ so that 
$a(t), b(t) \in \mathcal M$. Let $q(t) = -a(t)-b(t)$. It is easy to check that $q(t)\in \mathcal M$.

Since $\mathcal R \cup \mathcal T$ is the set of all minimal vectors of $\Esix$, the Voronoi parallelotope $P_V(\Esix)$ has no standard faces other than
its facets and $m$-edges.

\section{Freedom of $P_V(\Esix)$}\label{SectionFreedom}

Recall that a vector $v$ is free for a parallelotope $P$ if the Minkowski sum $P+z(v)$ is also a parallelotope.

\begin{proposition}\label{tri}
There are the following $120$ triples generating $6$-belts of the Voronoi
polytope $P_V(\Esix)$, where facet vectors are given in the basis
$\mathcal E$, $e(I_6)=\sum_{i\in I_6}e_i$, $e(S)=\sum_{i\in S}e_i$ and $S$
is a subset of $I_6$ of cardinality $3$.
\begin{itemize}
\item[(i)] $e_i-e_j, e_j-e_k, e_k-e_i$, $i,j,k\in I_6$;

\item[(ii)] $e(S), e({\overline S}), e(I_6)$, ${\overline S}=I_6-S$;

\item[(iii)] $e(S), e_i-e_j, e(S)-e_i+e_j$, $i\in S, j\not\in S$.
\end{itemize}
\end{proposition}
\proof It is easy to verify that there are $20$ triples of type
(i), $10$ triples of type (ii) and 90 triples of type (iii), total $120$
triples.
Each $6$-belt is uniquely determined by each of its six $4$-faces. It is
known (see, for example \cite{CS}) that $P_V(\Esix)$ has $720$
$4$-faces. Since each $4$-face belongs exactly to one $6$-belt,
$P_V(\Esix)$ has $\frac{720}{6}=120$ $6$-belts. \qed

Note that the above $6$-belts form $1$-orbit under the action of the
automorphism group of $\Esix$.

\vspace{2mm}
Call a vector {\em free for a triple} if it is orthogonal at least to one
vector of the triple.
\begin{theor}\label{27a}
The Voronoi polytope $P_V(\Esix)$ is free along a line $l$ if and
only if $l$ is spanned by a minimal vector $q\in{\mathcal M}$ of the dual
lattice $\Esix^*$, described in \eqref{min}.
\end{theor}
\proof We seek a vector $a$ which is free for $P_V(\Esix)$ in the
basis ${\mathcal A}=\{a_i:i\in I_6\}$ related to the lattice
$\Esix^*$ which is dual to the basis ${\mathcal E}=\{e_i:i\in I_6\}$
related to the lattice $\Esix$. So, let $a=\sum_{i\in I_6}z_ia_i$ be
a vector which is free for $P_V(\Esix)$. We find conditions, when the
vector $a$ is orthogonal to at least one vector of each triple of types (i), (ii) and (iii) of Proposition~\ref{tri}. We shall see that $a$ is, up to a
multiple, one of the vectors $a_i,b_j,c_{kl}$, $i,j,k,l\in I_6$, of
\eqref{abc}.

{\bf Claim 1}. {\em The coordinates $z_i, i\in I_6$ take only two values}.
Suppose that there are three pairwise distinct coordinates
$z_i\not=z_j\not=z_k\not=z_i$. Then the vector $a$ is not free for a triple
of type (i).

So, a free vector has the form $a_T=za(T)+z'a({\overline T})$, where
$T\subseteq I_6$, ${\overline T}=I_6-T$ and $a(T)=\sum_{i\in T}a_i$.

{\bf Claim 2}. {\em If $z=0$ or $z'=0$, then $|{\overline T}|=1$, or
$|T|=1$, respectively}. In fact, if $z'=0$ and $|T|\ge 2$, then the vector
$a=za(T)$ is not free for each triple of type (ii) such that
$S\cap T\not=\emptyset$ and $S\cap\overline{T}\not=\emptyset$.

Note that Claim 2 implies that $T\not=\emptyset$ and $T\not=I_6$.

It is easy to verify that each of the six vectors $a_i$, $i\in I_6$, (which
are minimal vectors of $\Esix^*$) is free for all triples.

Now consider vectors $a_T=za(T)+z'a({\overline T})$, with both non-zero
coefficients $z$ and $z'$.

{\bf Claim 3}. {\em $|T|\not=3$}. In fact, let $|T|=3$. Consider a triple of
type (ii) for $S=T$. The vector $a_T$ is free for this triple only if
$e(I_6)^Ta_T=z+z'=0$. Hence, $a_T$ should take the form
$a_T=z(a(T)-a({\overline T}))$. But this vector is not free for a triple
of type (iii) with $S=T$ and $i\in T$, $j\not\in T$.

So, without loss of generality we can consider vectors $a_T$ such that
$|T|=1,2$.

{\bf Claim 4}. {\em $z=-2z'$}. Let $|T|=1$. For $a_T$ to be free for a
triple of type (ii) with $S\supset T$, the coefficients $z,z'$ should
satisfy one of the equalities $z+2z'=0$ or $z+5z'=0$. If $z=-5z'$, then
$a_T$ is not free for a triple of type (iii) such that $\{i\}=T$. Hence,
$z=-2z'$. It is easy to verify that, for all $i\in I_6$, the vector
$a_T=-z'(2a_i-a(I_6-\{i\}))=-3z'(a_i-\frac{1}{3}a(I_6))=-3z'b_i$ is free
for all triples.

Now, let $|T|=2$. For $a_T$ to be free for a triple of type (ii) with
$S\supset T$, the coefficients $z$ and $z'$ should satisfy one of the
equalities $2z+z'=0$ or $2z+4z'=0$. If $z'=-2z$, then the vector
$a_T=z(a(T)-2a({\overline T}))$ is not free for a triple of type (iii)
such that $S\cap T=\{i\}$ and $j\not\in T$. One can verify that if
$z=-2z'$ and $T=\{ij\}$, then, for $1\le i<j\le 6$, the vector
$a_T=3z'(a_i+a_j-\frac{1}{3}a(I_6))=3z'c_{ij}$ is free for triples of
all types. \qed

\section{A sum of $P_V(\Esix)$ with a zonotope}\label{SectionE6zonotope}

Our main goal is to classify all subsets $U \subset \mathcal M$ such that $P+Z(U)$ is a parallelotope. If $P+Z(U)$ is a parallelotope,
we call $U$ {\it feasible}, otherwise $U$ is called {\it forbidden}.

We say that that $U_0\subset \mathcal M$ is a {\it minimal forbidden set} for $P_V(\Esix)$, if $P_V(\Esix) + Z(U_0)$ is not a parallelotope, 
but $P_V(\Esix) + Z(U_1)$ is a parallelotope for every $U_1 \subsetneq U_0$. Obviously, every forbidden set $U \subset \mathcal M$ contains some minimal 
forbidden subset $U_0$.

Let $U_0\subset \mathcal M$ be a minimal forbidden set for $P_V(\Esix)$.
Choose an arbitrary element $p_0 \in U_0$ and let $U_1 = U_0 \setminus \{p_0\}$ be the {\it pre-forbidden set}. Notice that $p_0$ is not a free
vector for the parallelotope $P_V(\Esix) + Z(U_1)$. Hence by Theorem~\ref{bfr} $P_V(\Esix) + Z(U_1)$ has a $6$-belt such that $p_0$ is not parallel to
any facet of that belt. This motivates us to study $6$-belts of $P_V(\Esix) + Z(U_1)$.

Lemma~\ref{lem:st_vec_hered} implies that each standard vector $s$ of $P_V(\Esix) + Z(U_1)$ can be represented in the form
$$s = s' + \sum\limits_{p\in U_1} 2n_p(s') p,$$
where $s'$ is a standard vector of $P_V(\Esix)$. In this situation we say that $s$ {\it corresponds} to $s'$ and vice versa. 

Thus every facet vector of $P_V(\Esix) + Z(U_1)$ arises from some standard vector of $P_V(\Esix)$. Moreover, from
Lemma~\ref{lem:depend_hered} follows that the 3 pairs vectors of each $6$-belt of $P_V(\Esix) + Z(U_1)$ arise from 3 pairs of coplanar standard vectors of
$P_V(\Esix)$.

So, for what follows $2$-dimensional subsets of the set ${\mathcal R}\cup{\mathcal T}$ are important. Each of these subsets generates a $2$-plane $\alpha$, and 
it is the intersection $({\mathcal R}\cup{\mathcal T}) \cap \alpha$. 

Now we state several lemmas that are necessary to describe all minimal forbidden sets for $P_V(\Esix)$. Their proofs are technical, and therefore transfered to
the Appendix B.

\begin{lem}\label{L2}
Let $\alpha$ be a two-dimensional plane such that the intersection  
$$(\mathcal R \cup \mathcal T) \cap \alpha$$
consists of at least 6 vectors, i.e., at least three pairs of antipodal vectors. Then this intersection is equivalent to one of the following 5 planar sets
up to the action of $W(\Esix)$. 

\begin{enumerate}
\item[(a)] $\{\pm(a_2 - a_1), \pm(b_1 - a_1), \pm(a_2 - b_1), \pm(b_2 - a_1)\}$; 
\item[(b)] $\{\pm(a_2 - a_1), \pm(a_3 - a_1), \pm (a_3 - a_2)\}$; 
\item[(c)] $\{\pm(b_2 - a_1), \pm(c_{12} - a_1), \pm(c_{12} - b_2)\}$; 
\item[(d)] $\{\pm(b_2 - a_1), \pm(b_4 - a_3), \pm(c_{23} - c_{14})\}$; 
\item[(e)] $\{\pm(b_2 - a_1), \pm(b_3 - a_1), \pm(b_3 - b_2)\}$. 
\end{enumerate}
\end{lem}

\begin{lem}\label{lem:transversal_to_m_edge}
Let $p \in \mathcal M$.
Then the direct Minkowski sum $[a_2, -b_1] \oplus z(p)$ is a standard 2-face of the parallelotope $P_V(\Esix) + z(p)$ iff
$$p \in \{ a_1, b_2, c_{34}, c_{35}, c_{36}, c_{45}, c_{46}, c_{56} \}.$$
\end{lem}

\begin{lem}\label{lem:six_belts}
Let $U \subset \mathcal M$ and let $P_V(\Esix) + Z(U)$ be a parallelotope. Then
\begin{enumerate}
\item[\rm 1.] $P_V(\Esix) + Z(U)$ necessarily has a 6-belt with facet vectors corresponding to 
$$\{\pm(a_2 - a_1), \pm(a_3 - a_1), \pm (a_3 - a_2)\}.$$
\item[\rm 2.] $P_V(\Esix) + Z(U)$ has a 6-belt with facet vectors corresponding to 
$$\{\pm(a_2 - a_1), \pm(b_1 - a_1), \pm(b_2 - a_1)\},$$ 
iff $U$ contains a subset $U' = \{p_1, p_2, p_3, p_4\}$, where
$$p_1 \in \{c_{34}, c_{56}\}, \; p_1 \in \{c_{35}, c_{46}\}, \; p_3 \in \{c_{36}, c_{45}\}, \; p_4 \in \{a_2, b_1\}.$$
\item[\rm 3.] $P_V(\Esix) + Z(U)$ has a 6-belt with facet vectors corresponding to 
$$\{\pm(b_2 - a_1), \pm(b_3 - a_1), \pm(b_3 - b_2)\},$$ 
iff $U$ contains a subset $\{b_1, c_{45}, c_{46}, c_{56}\}$.
\item[\rm 4.] $P_V(\Esix) + Z(U)$ has no $6$-belt with facet vectors corresponding to
$$\{\pm(a_2 - a_1), \pm(b_1 - a_2), \pm(b_2 - a_1)\}.$$
\item[\rm 5.] $P_V(\Esix) + Z(U)$ has no $6$-belt with facet vectors corresponding to
$$\{\pm(b_2 - a_1), \pm(c_{12} - a_1), \pm(c_{12} - b_2)\}.$$
\item[\rm 6.] $P_V(\Esix) + Z(U)$ has no $6$-belt with facet vectors corresponding to
$$\{\pm(b_2 - a_1), \pm(b_4 - a_3), \pm(c_{23} - c_{14})\}.$$
\end{enumerate}
\end{lem}

Now we are ready to state and prove the main result.

\begin{theor}
\label{PUP}
Let $U \subset \mathcal M$. The following assertions are equivalent.
\begin{enumerate}
\item[(i)] There exists a subset $U' \subseteq U$ such that $|U'| = 5$ and $p^T p' = 1/3$ for every $p, p' \in U'$, $p\neq p'$.
\item[(ii)] The polytope $P_V(\Esix)+Z(U)$ is not a parallelotope.
\end{enumerate}
\end{theor}

\proof (i)$\Rightarrow$(ii) Assume the converse, namely, that $P_V(\Esix)+Z(U)$ is a parallelotope. Then $P_V(\Esix)+Z(U_0)$ is a parallelotope as well. 

By Lemma~\ref{TJ}, $U_0 = T_J(q)$ for some $q\in \mathcal M$. Without loss of generality, let $q = c_{12}$. Then $U_0 = \{p_1, p_2, p_3, p_4, p_5\}$, where
$$p_1 \in \{c_{34}, c_{56}\}, \; p_2 \in \{c_{35}, c_{46}\}, \; p_3 \in \{c_{36}, c_{45}\}, \; p_4 \in \{a_2, b_1\}, p_5 \in \{a_1, b_2\}.$$

Let $U_1 = \{p_1, p_2, p_3, p_4 \}$, $U_2 = \{p_1, p_2, p_3, p_5 \}$. Since $P_V(\Esix)+Z(U_0)$ is a parallelotope, so are $P_V(\Esix)+Z(U_1)$ and
$P_V(\Esix)+Z(U_2)$. 

By Lemma~\ref{lem:six_belts}, Assertion 2, the parallelotope $P_V(\Esix)+Z(U_1)$ has a 6-belt with facet vectors corresponding to 
$$\{\pm(a_2 - a_1), \pm(b_1 - a_1), \pm(b_2 - a_1)\}.$$
On the other hand, the parallelotope $P_V(\Esix)+Z(U_2)$ has a 6-belt with facet vectors corresponding to
$$\{\pm(a_2 - a_1), \pm(b_1 - a_1), \pm(b_1 - a_2)\}.$$
Indeed, this statement is analogous to Lemma~\ref{lem:six_belts}, Assertion 2 up to interchange of $a_1$ and $a_2$. 

Hence $P_V(\Esix)+Z(U)$ has two 6-belts with
4 common facets (these belts are inherited from $P_V(\Esix)+Z(U_1)$ and $P_V(\Esix)+Z(U_2)$). But 2 different belts can share only 2 facets.
The contradiction shows that $P_V(\Esix)+Z(U)$ is not a parallelotope.

(ii)$\Rightarrow$(i) Suppose that $P_V(\Esix)+Z(U)$ is not a parallelotope. Then $U$ contains some minimal forbidden subset, say $U_0\subseteq U$.
Choose arbitrarily $p \in U_0$ and let $U_1 = U_0 \setminus \{p\}$.
Then $P_V(\Esix)+Z(U_1)$ is a parallelotope and $P_V(\Esix)+Z(U_0)+z(p)$ is not.
Thus $P_V(\Esix)+Z(U_0)$ has a $6$-belt with no facet of this belt parallel to $p$. According to Lemmas~\ref{L2}~and~\ref{lem:six_belts}, the 6-belt
is of one of the three types.

\begin{enumerate}
\item The facet vectors of the 6-belt correspond to a configuration of vectors from $\mathcal R \cup \mathcal T$ that is equivalent to 
$$\{\pm(a_2 - a_1), \pm(a_3 - a_1), \pm(a_3 - a_2)\}$$
up to an isometry of $P_V(\Esix)$.
\item The facet vectors of the 6-belt correspond to a configuration of vectors from $\mathcal R \cup \mathcal T$ that is equivalent to
$$\{\pm(a_2 - a_1), \pm(b_1 - a_1), \pm(b_2 - a_1)\}$$ 
up to an isometry of $P_V(\Esix)$.
\item The facet vectors of the 6-belt correspond to a configuration of vectors from $\mathcal R \cup \mathcal T$ that is equivalent to
$$\{\pm(b_2 - a_1), \pm(b_3 - a_1), \pm(b_3 - b_2)\}$$
up to an isometry of $P_V(\Esix)$.
\end{enumerate}

Consider these three cases separately.

\noindent{\bf Case 1.} Without loss of generality, assume that the facet vectors of the 6-belt correspond exactly to the 6-tuple 
$$\{\pm(a_2 - a_1), \pm(a_3 - a_1), \pm(a_3 - a_2)\}.$$
Then the facets of this belt are parallel to the hyperplanes
$$\left\langle a_2 - a_1 \right\rangle^\bot, \left\langle a_3 - a_1 \right\rangle^\bot, \left\langle a_3 - a_2 \right\rangle^\bot.$$
These are exactly the facet directions of one particular $6$-belt of $P_V(\Esix)$. Since $p$ is free for $P_V(\Esix)$, it is parallel to one of those facets.
This contradicts the assumption on the vector $p$, so this case is impossible.

\noindent{\bf Case 2.} Without loss of generality, assume that the facet vectors of the 6-belt correspond exactly to the 6-tuple 
$$\{\pm(a_2 - a_1), \pm(b_1 - a_1), \pm(b_2 - a_1)\}.$$
Then the 4-dimensional direction of this 6-belt is 
$$\left\langle a_2 - a_1, b_1 - a_1\right\rangle^{\bot},$$
because two pairs of facets of the belt are extended facets of $P_V(\Esix)$ and have the same direction with those facets. Thus the orthogonal projection
of $P_V(\Esix) + Z(U_0)$ onto the 2-plane $\left\langle a_2 - a_1, b_1 - a_1\right\rangle$ should be an 8-gon. Similarly to the proof of
Assertion~2 of Lemma~\ref{lem:six_belts}, we conclude that $U_0$ should contain a vector $p_5 \in \{a_1, b_2\}$ (as well as $a_2$ or $b_1$). 
By the same Assertion~2 of Lemma~\ref{lem:six_belts}, $U_1$ also contains the set $\{p_1, p_2, p_3, p_4\}$, where
$$p_1 \in \{c_{34}, c_{56}\}, \; p_1 \in \{c_{35}, c_{46}\}, \; p_3 \in \{c_{36}, c_{45}\}, \; p_4 \in \{a_2, b_1\}.$$
Setting $U' = \{p_1, p_2, p_3, p_4, p_5\}$ finishes the proof of this case, because $p_i^Tp_j = 1/3$ for every $1\leq i < j \leq 5$. The last identity follows,
for instance, from the fact that $U'$ is of the form $T_J(a_2 - a_1)$ for every possible choice of its elements.

\noindent{\bf Case 3.} Without loss of generality, assume that the facet vectors of the 6-belt correspond exactly to the 6-tuple 
$$\{\pm(b_2 - a_1), \pm(b_3 - a_1), \pm(b_3 - b_2)\}.$$

By Assertion 3 of Lemma~\ref{lem:six_belts}, the set $U_0$ contains a subset $\{b_1, c_{45}, c_{46}, c_{56}\}$. The proof of the same assertion also
shows that the facets of the 6-belt are orthogonal to the vectors $b_2 - a_1, b_3 - a_1, b_3 - b_2$. Therefore the vector $p$, which is not parallel
to any of those facets, is not orthogonal to any of the vectors $b_2 - a_1, b_3 - a_1, b_3 - b_2$. This is possible only for $p = b_2$ or $p = b_3$.

Setting $U' = \{b_1, b_2, c_{45}, c_{46}, c_{56}\}$ or $U' = \{b_1, b_3, c_{45}, c_{46}, c_{56}\}$ correspondingly, we finish the proof of this case.

Since every possible case is considered, the implication (ii)$\Rightarrow$(i) is proved. \qed

Recall that a subset $U \subset \mathcal M$ is {\it feasible} if $P_V(\Esix) + Z(U)$ is a parallelotope. As mentioned before, all feasible subsets are
unimodular. Hence the following corollary.

it contains no subset $U' \subseteq U$ such that $|U'| = 5$ and $p^T p' = 1/3$ for every 
$p, p' \in U'$, $p\neq p'$. If $U$ is feasible, then $P_V(\Esix) + Z(U)$ is a parallelotope.
This immediately gives the following corollary.

\begin{cor}\label{uni}
Suppose that $U\subseteq\mathcal M$, and there is no subset $U' \subseteq U$ such that $|U'| = 5$ and $p^T p' = 1/3$ for every 
$p, p' \in U'$, $p\neq p'$. Then $U$ is unimodular.
\end{cor}

\section{Maximal feasible subsets in $\mathcal M$}\label{SectionEnumeration}

It is convenient to denote a unimodular system $U$ by a regular matroid
$M_U$, which is represented by $U$. There are many definitions of a matroid,
see, for example, \cite{Aig}. In particular, a matroid on a ground set $X$
is a family $\mathcal C$ of {\em circuits} $C\subseteq X$ satisfying the
following axioms:
\begin{itemize}
\item If $C_1,C_2\in\mathcal C$, then $C_1\not\subseteq C_2$, $C_2\not\subseteq C_1$
\item and if $x\in C_1\cap C_2$, then there is $C_3\in\mathcal C$ such that $C_3\subseteq C_1\cup C_2 - \{x\}$.
\end{itemize}
Note that linear dependencies between vectors of any family of vectors
satisfy these axioms.
A matroid $M$ on a set $X$ is {\em represented} by a set of vectors $U$ if
there is a one-to-one map $f:X\to U$ such that, for all $C\in\mathcal C$, the
set $f(C)$ is a minimal by inclusion linearly dependent subset of $U$.

Each unimodular set of vectors represents a {\em regular} matroid. Special
cases of a regular matroids are {\em graphic} $M(G)$ and {\em cographic}
$M^*(G)$ matroids whose ground sets are the set of edges of a graph $G$. The
families of circuits of these matroids are {\em cycles} and {\em cuts} of
the graph $G$, respectively.

Seymour proved in \cite{Se} that a regular matroid is a $1$-, $2$- and
$3$-sum of graphic, cographic matroids and a special matroids $R_{10}$,
which is neither graphic nor cographic (see also \cite{tru}). Using this
work of Seymour, the authors of \cite{DaG} described maximal by inclusion
unimodular systems of a given dimension. Their description is similar to
that of Seymour, but they denote a $k$-sum of Seymour by $(k-1)$-sum, and
slightly changed the definition of the $k$-sum. Namely, $k$-sum of Seymour
gives the symmetric difference of the summing sets, but the corresponding
$(k-1)$-sum of \cite{DaG} gives the union of the summing sets. Below, we use
the second $k$-sum from \cite{DaG}.

The graphic matroid $M(\KnP)$ of the complete graph
$\KnP$ on $n+1$ vertices is represented by the root system
$\An$, which is one of maximal by inclusion unimodular
$n$-dimensional systems. The matroid $R_{10}$ is represented by a
ten-element unimodular system of dimension $5$, which is denoted in
\cite{DaG} by $\Efive$. Each maximal by inclusion cographic unimodular
system of dimension $n$ represents the cographic matroid $M^*(G)$ of a
$3$-connected cubic non-planar graph $G$ on $2(n-1)$ vertices and $3(n-1)$
edges.

All ten maximal by inclusion unimodular feasible subsets
$U\subseteq\mathcal M$ were found by computer. The method is to take an
unimodular subset, to consider all possible ways to extend it in an
acceptable way and to reduce by isomorphism. The results are presented on
Table \ref{AllTenSystem}. We give vectors of these sets in the denotations
\eqref{min}. Let $\Aut({\mathcal M})$ be the group of automorphisms of the
set ${\mathcal M}$. We denote by $\Stab(U)\subseteq \Aut({\mathcal M})$ the
stabilizer subgroup of the set $U$. Among these $10$ sets there are two
absolutely maximal $6$-dimensional unimodular sets $U_3$ and $U_{10}$
(see \cite{DaG}). The matroids which are neither graphic nor cographic are
given in Table \ref{NonGraphicCographicMatroids} and the rest in
Figure \ref{GraphicOrCographic}.

\begin{table}
\begin{center}
\begin{tabular}{|c|c|c|c|c|c|}
\hline
nr & $|U|$ & $\dim U$ & $|Stab\, U|$ & status & ref\\
\hline
1 &  9  & 5 & 384 &   graphic & $M(G_1)$\\
2 & 12  & 5 &  96 & cographic & $M^*(G_2)$\\
3 & 12  & 6 &  12 &   special & $R_{12}$\\
4 & 12  & 6 &  12 &   special & $R_{10}\oplus_1 C_3$\\
5 & 13  & 6 &   4 & cographic & $M^*(G_5)$\\
6 & 13  & 6 &   4 & cographic & $M^*(G_6)$\\
7 & 13  & 6 &   2 & cographic & $M^*(G_7)$\\
8 & 14  & 6 &   8 &   graphic & $M(G_8)$\\
9 & 14  & 6 &  24 &   graphic & $M(G_9)$\\
10& 15  & 6 &  24 & cographic & $M^*(G_{10})$\\
\hline
\end{tabular}
\end{center}
\caption{All ten maximal admissible unimodular system of ${\mathcal M}$}
\label{AllTenSystem}
\end{table}

\begin{table}
\begin{center}
\begin{minipage}[b]{5.5cm}
\centering
\begin{equation*}
\begin{array}{|c|cccccc|}
\hline
b_2    & 1 & 0 & 0 & 0 & 0 & 0 \\
c_{26} & 0 & 1 & 0 & 0 & 0 & 0 \\
a_2    & 0 & 0 & 1 & 0 & 0 & 0 \\
a_3    & 0 & 0 & 0 & 1 & 0 & 0 \\
b_1    & 0 & 0 & 0 & 0 & 1 & 0 \\
b_5    & 0 & 0 & 0 & 0 & 0 & 1 \\ \hline
b_3    & 1 & 0 &-1 & 1 & 0 & 0 \\
c_{36} & 0 & 1 & 1 &-1 & 0 & 0 \\
a_4    & 0 & 1 & 0 &-1 &-1 &-1 \\
a_1    &-1 & 0 & 1 & 0 & 1 & 0 \\
c_{14} & 0 &-1 & 0 & 1 & 0 & 1 \\
c_{15} & 1 & 0 &-1 & 0 &-1 &-1 \\ \hline
\end{array}
\end{equation*}
The unimodular system $R_{12}$
\end{minipage}
\begin{minipage}[b]{5.5cm}
\centering
\begin{equation*}
\begin{array}{|c|cccccc|}
\hline
b_2    & 1 & 0 & 0 & 0 & 0 & 0 \\
c_{26} & 0 & 1 & 0 & 0 & 0 & 0 \\
a_2    & 0 & 0 & 1 & 0 & 0 & 0 \\
a_3    & 0 & 0 & 0 & 1 & 0 & 0 \\
b_1    & 0 & 0 & 0 & 0 & 1 & 0 \\
b_5    & 0 & 0 & 0 & 0 & 0 & 1 \\  \hline
b_3    & 1 & 0 &-1 & 1 & 0 & 0 \\
c_{36} & 0 & 1 & 1 &-1 & 0 & 0 \\
a_4    & 0 & 1 & 0 &-1 &-1 &-1 \\
a_1    &-1 & 0 & 1 & 0 & 1 & 0 \\
c_{16} & 1 & 1 & 0 & 0 &-1 & 0 \\
c_{45} & 0 &-1 & 0 & 1 & 1 & 0 \\ \hline
\end{array}
\end{equation*}
The unimodular system $R_{10}\oplus_1 C_3$
\end{minipage}

\end{center}
\caption{Unimodular systems which are neither graphic nor cographic}
\label{NonGraphicCographicMatroids}
\end{table}

\begin{figure}

\begin{center}
\begin{minipage}[b]{5.5cm}
\centering
\input{K12_e.pstex_t}\\
$G_1$
\end{minipage}
\begin{minipage}[b]{5.5cm}
\centering
\input{Ggamma.pstex_t}\\
$G_2$
\end{minipage}
\begin{minipage}[b]{5.5cm}
\centering
\input{Ga.pstex_t}\\
$G_5$
\end{minipage}
\begin{minipage}[b]{5.5cm}
\centering
\input{Gb.pstex_t}\\
$G_6$
\end{minipage}
\begin{minipage}[b]{5.5cm}
\centering
\input{Case7_G7.pstex_t}\\
$G_7$
\end{minipage}
\begin{minipage}[b]{5.5cm}
\centering
\input{Case8_G1.pstex_t}\\
$G_8$
\end{minipage}
\begin{minipage}[b]{5.5cm}
\centering
\input{Case9_G1.pstex_t}\\
$G_9$
\end{minipage}
\begin{minipage}[b]{5.5cm}
\centering
\input{Case10_G10.pstex_t}\\
$G_{10}$
\end{minipage}
\end{center}
\caption{Occurring graphs}
\label{GraphicOrCographic}
\end{figure}
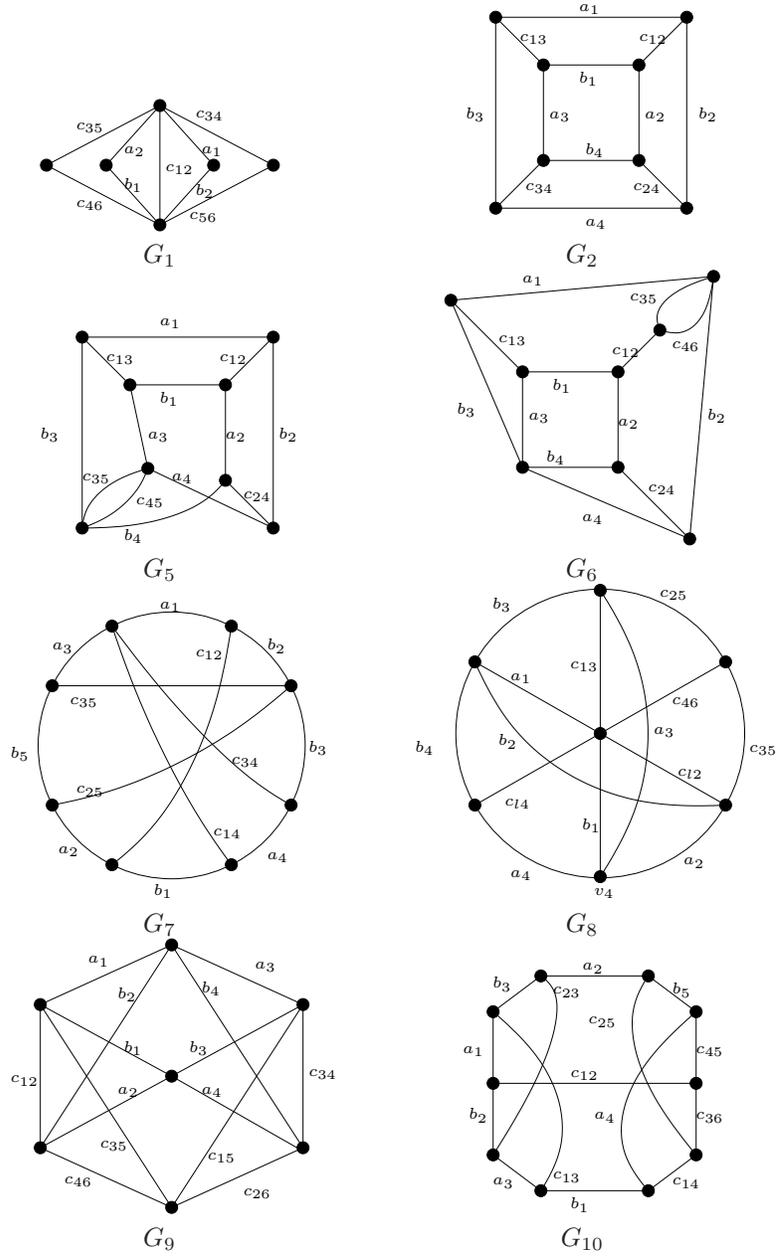

\section{Results for other lattices}\label{SectionOtherLattice}
The algorithms explained in Section \ref{SectionEnumAlgo}
have been implemented in \cite{polyhedral} and applied to several
highly symmetric lattices.
We thus proved that the polytopes $P_V(L)$ are nonfree
for $L=\mathsf{E}_7^*$, $\kappa_8^*$, $\kappa_9^*$, $\Lambda_{10}^*$
or $K_{12}$.
For the lattice $BW_{16}$ this direct approach does not work, because
there are too many $6$-belts. However, by selecting a subset of the
$6$-belts, one can prove that this lattice is nonfree as well.

For a general lattice $L$, we cannot expect a simple criterion
for determining the forbidden and feasible subsets of $P_V(L)$.
This is because the lattice $\Esix$ is very symmetric and other cases
are necessarily more complicated.

In Table \ref{FreeInformationForLattices} we give detailed information
on $12$ lattices of the computed data.
For the lattice $ER_7$, we found out that there exist some minimal
forbidden subsets $U$ such that $P_V(ER_7)$ has some faces $G$ such
that $\dim\, G_U = 7$.
Hence, for this case, one cannot limit oneself to the quasi $4$-belt
in the analysis.

\begin{table}
\begin{tabular}{|r|c|c|c|c|c|}
\hline
$L$    & $|{\mathcal F}(L)|$ & $|O({\mathcal F}_{min}(L))|$ & $|O({\mathcal F}_{max}(L))|$ & $\dim_{max}(L)$ & $S_{max}(L)$\\
\hline
$\Esix$      & 27 & 2  & 10 & 6 & 15\\
$\Eseven$    & 28 & 2  & 4  & 7 & 14\\
$ER_7$       & 28 & 9 & 49 & 7 & 15\\
$ER_7^{*}$    & 6 & 0 & 1 & 6 & 6\\
$\kappa_7$   & 11 & 2 & 2 & 6 & 9\\
$\kappa_7^*$ & 10 & 2 & 1 & 1 & 1\\
$\kappa_8$   & 6 & 1 & 2 & 3 & 3\\
$\kappa_9$   & 3 & 1 & 1 & 2 & 2\\
$\Lambda_9$  & 1 & 0 & 1 & 1 & 1\\
$\Lambda_9^*$ & 16 & 2 & 1 & 1 & 1\\
$\Lambda_{10}$ & 3 & 0 & 1 & 2 & 3\\
$O_{10}$       & 40 & 2 & 2 & 5 & 10\\
\hline
\end{tabular}
\caption{For $12$ lattices, we give the following information on their free structure: $|{\mathcal F}(L)|$ is the size of ${\mathcal F}(L)$, 
$|O({\mathcal F}_{min}(L)|$ is the number of orbits of minimal forbidden subsets, 
$|O({\mathcal F}_{max}(L)|$ is the number of orbits of maximal unimodular subsets, 
$\dim_{max}(L)$ is the maximum dimension of admissible subsets of ${\mathcal F}(L)$,
$S_{max}(L)$ is the maximum size of an admissible subsets of ${\mathcal F}(L)$.
The lattices are given in \cite{nebe_sloane} and the lattice $ER_7$ is given in \cite{35tope}.}
\label{FreeInformationForLattices}
\end{table}

\section*{Appendix A. Feasible extensions of $(n-3)$-faces}
Hereby we consider feasible extensions of $(n-3)$-faces of the parallelotope $P$ and construct local matchings around these extensions.
The classification of feasible arrangements of $\Fan(G)$ and $\pi_G(z(v))$ is given in Figures~\ref{segm1} and~\ref{segm2}.

We will use the notation $v \in U_i(G)$ ($i = 1,2,3$) introduced in Section~\ref{sect:Intro}. Recall the meaning of the notation. 
If $\lambda > 0$ is sufficiently small, and $x$ is an arbitrary point in the relative interior of $G$, then we write 
\begin{itemize}
\item $v\in U_1(G)$ if $x\pm\lambda v\notin G$ and one of the points $x\pm \lambda v$ is in $P$.
\item $v\in U_2(G)$ if $x\pm \lambda v\in \Lin(G)$.
\item $v\in U_3(G)$ if $x\pm \lambda v\notin G$ and $x\pm\lambda v\notin P$.
\end{itemize}
A feasible extension of the gace $G$ exists iff $v\in U_3(G)$.

If $\Fan(G)$ and $\pi_G(z(v))$ are arranged together as in one of the cases a.1), a.2), b.2), c.2), c.3), d.2), d.3), e.3),
then $v \in U_1(G)$, regardless of the position of $\pi_G(P)$ relative to $\Fan(G)$. Hence $P+z(v)$ has no 
$(n-2)$-face of the form $G+z(v)$. Consider the remaining 5 cases.

\noindent{\bf b.1)} Suppose that $\Fan(G)$ and $\pi_G(z(v))$ are arranged together as in case b.1). Then, if $\pi_G(P)$ lies in one of the cones 
$O(-x_1)(-x_4)x_2x_3$ or $Ox_1x_4(-x_2)(-x_3)$, we have $v \in U_1(G)$. In all four other subcases we have $v \in U_3(G)$.

Let $P$, $P_1$, $P_2$, and $P_3$ be the translates of $P$ that share the face $G$ so that their projections along $G$ lie in the four cones 
$Ox_1x_2x_3x_4$, $O(-x_1)(-x_2)x_4x_3$, $O(-x_1)(-x_2)(-x_3)(-x_4)$, and $Ox_1x_2(-x_4)(x_3)$. Then $P + z(v)$, $P_1 + z(v)$, $P_2 + z(v)$, and 
$P_3 + z(v)$ match around their common face $G+z(v)$. This matching can be presented as splitting $D(G)$ into two quadrangular pyramids by a 
parallelogram, which is the ``dual cell'' of $G+z(v)$. We do not assume that the tiling by translates $P+z(v)$ exists, however, the local 
matching proves that $G+z(v)$ determines a $4$-belt.

\noindent{\bf c.1)} Suppose that $\Fan(G)$ and $\pi_G(z(v))$ are arranged together as in case c.1). Then, if $\pi_G(P)$ lies in one of the cones 
$Ox_1x_4y$ or $Ox_2x_3y$,  we have $v \in U_1(G)$. For the other three subcases we have $v \in U_3(G)$.

Let $P$, $P_1$, and $P_2$ be the translates of $P$ that share the face $G$ so that their projections along $G$ lie in the three cones 
$Ox_1x_2x_3x_4$, $Ox_1x_2y$, and $Ox_3x_4y$. Then $P + z(v)$, $P_1 + z(v)$, and $P_2 + z(v)$ match around their common face $G+z(v)$. This matching 
can be presented as splitting the dual cell $D(G)$ into two tetrahedra by a triangle, which is the ``dual cell'' of $G+z(v)$. Thus $G+z(v)$ determines a $6$-belt.

\noindent{\bf d.1)} Suppose that $\Fan(G)$ and $\pi_G(z(v))$ are arranged together as in case d.1). Then, if $\pi_G(P)$ lies in one of the three cones 
$Ox_1x_2y$, $Ox_1x_3y$ or $Ox_2x_3(-y)$, then we have $v \in U_1(G)$. For the other three subcases we have $v \in U_3(G)$.

Let $P$, $P_1$, and $P_2$ be the translates of $P$ that share the face $G$ so that their projections along $G$ lie in the three cones 
$Ox_1x_2(-y)$, $Ox_1x_3(-y)$ or $Ox_2x_3y$,. Then $P + z(v)$, $P_1 + z(v)$, and $P_2 + z(v)$ match around their common face $G+z(v)$. This matching can be 
presented as splitting $D(G)$ into a tetrahedron and a quadrangular pyramid by a triangle, which is the ``dual cell'' of $G+z(v)$. 

\noindent{\bf e.2)} Suppose that $\Fan(G)$ and $\pi_G(z(v))$ are arranged together as in case e.2). Then, if $\pi_G(P)$ lies in one of the four cones 
$Ox_1x_2x_3$, $Ox_1x_2(-x_3)$, $O(-x_1)(-x_2)x_3$ or $O(-x_1)(-x_2)(-x_3)$, 
then we have $v \in U_1(G)$. For the other three subcases we have $v \in U_3(G)$.

Let $P$, $P_1$, $P_2$, and $P_3$ be the translates of $P$ that share the face $G$ so that their projections along $G$ lie in the four cones 
$Ox_1(-x_2)x_3$, $Ox_1(-x_2)(-x_3)$, $O(-x_1)x_2x_3$, and $O(-x_1)x_2(-x_3)$. 
Then $P + z(v)$, $P_1 + z(v)$, $P_2 + z(v)$, and $P_3 + z(v)$ match around their common face $G+z(v)$. 
This matching can be presented as splitting $D(G)$ into two triangular prisms by a parallelogram, which is the ``dual cell'' of $G+z(v)$.

\noindent{\bf e.1)} Suppose that $\Fan(G)$ and $\pi_G(z(v))$ are arranged together as in case e.1). Then, if $\pi_G(P)$ lies in one of the cones
$O(-x_1)x_2x_3$ or $Ox_1(-x_2)(-x_3)$, we have $v \in U_1(G)$. For the other three subcases we have $v \in U_3(G)$.

Let $P$, $P_1$, $P_2$, \ldots $P_5$ be the translates of $P$ that share the face $G$ so that their projections along $G$ lie in the six cones of
$\Fan(G)$ not mentioned in the previous paragraph. Let $P_6$ and $P_7$ be the other two translates of $P$ the face $G$.

Without loss of generality, we can assume that $P_6$ shares a facet with $P$, $P_1$ and $P_2$. Then $P_7$ shares a facet with $P_3$, $P_4$ and $P_5$.
We can also assume that the vector $\pi_G(v)$ is directed from $\pi_G(P_6)$ to $\pi_G(P_7)$. Then $P + z(v)$, $P_1 + z(v)$, and $P_2 + z(v)$ match around their 
common face $G+z(v)$, as well as $P_3 + 2v + z(v)$, $P_4 + 2v + z(v)$, $P_5 + 2v + z(v)$ match around $G + 2v + z(v)$ and share common facets with
the three polytopes mentioned just before. Thus $G$ produces two $(n-2)$-faces of the local matching. This can be presented as splitting $D(G)$ into 
two tetrahedra and an octahedron by two triangles, being the ``dual cells'' for $G+z(v)$ and $G + 2v + z(v)$.

\section*{Appendix B. 6-belts of parallelotopes $P_V(\Esix) + Z(U)$}

\begin{proof}[Proof of Lemma~\ref{L2}]
Notice that for every $r\in{\mathcal R}$ and $t\in{\mathcal T}$ one has $|r| = \sqrt{2}$ and $|t| = 2$.

Since every vector of $\mathcal R \cup \mathcal T$ is minimal for the lattice $\Esix$, then every vector from the set $(\mathcal R \cup \mathcal T) \cap \alpha$
is minimal for the two-dimensional lattice $\Esix \cap \alpha$.

Now we have two possibilities --- either the Delaunay tessellation 
$\mathcal D(\Esix \cap \alpha)$ for the lattice $\Esix \cap \alpha$ is triangular, or it is rectangular.

\noindent{\bf Case 1.} $\mathcal D(\Esix \cap \alpha)$ is triangular. Then there is no vector in $(\mathcal R \cup \mathcal T) \cap \alpha$,
except of the edge vectors of $\mathcal D(\Esix \cap \alpha)$. 
But, since the intersection $(\mathcal R \cup \mathcal T) \cap \alpha$ consists of at least 3 pairs of antipodal vectors, it should
contain every edge vector of $\mathcal D(\Esix \cap \alpha)$. As a consequence, $(\mathcal R \cup \mathcal T) \cap \alpha$ contains
three vectors that form a Delaunay triangle for $\Esix \cap \alpha$.

Every Delaunay triangle is acute-angled. Among all triangles that have all side-lengths equal to $\sqrt{2}$ or $2$, the acute-angled are the following:
\begin{itemize}
	\item equilateral with side-length $\sqrt{2}$;
	\item equilateral with side-length $2$.
	\item isosceles with sides equal to $2$, $2$, and $\sqrt{2}$.
\end{itemize}

This forces us to restrict the search to the following subcases:
\begin{itemize}
	\item[1.1.] $r,r'\in{\mathcal R}$ with $r^T r' = 1$ and $r - r' \in \mathcal R$;
	\item[1.2.] $t,t'\in{\mathcal T}$ with $t^T t' = 2$ and $t - t' \in \mathcal T$;
	\item[1.3.] $t,t'\in{\mathcal T}$ with $t^T t' = 3$ and $t - t' \in \mathcal R$;
\end{itemize}

Consider these subcases separately.

\noindent{\bf Subcase 1.1.} Consider a lattice triangle $\Delta$ in $\Esix$ with edge vectors $r$, $r'$ and $r - r'$. This triangle is equilateral with 
side-length $\sqrt{2}$. We prove that $\Delta$ is a Delaunay triangle for $\Esix$. 

Consider a ball $B \subset \mathbb R^6$ centered at the center of $\Delta$ and containing the vertices of $\Delta$ on its boundary. 
$B$ contains no points of $\Esix$ other than the vertices of $\Delta$. Indeed, every point of $B$ is at least as close to the nearest vertex of $\Delta$
as $2 / \sqrt{3}$. But the minimum distance between points of $\Esix$ is $\sqrt{2} > 2 / \sqrt{3}$, a contradiction.

Since $\Delta$ is a Delaunay triangle for $\Esix$, it is a face of $P_{Schl}$. All such faces are equivalent to $\conv \{a_1, a_2, a_3\}$ up to some action
of $W(\Esix)$. Hence Subcase 1.1 gives exactly the planar sets of type (b).

\noindent{\bf Subcase 1.2.} In this case we write $t = p_1 - p_1'$, $t' = p_2 - p_2'$, where $p_i, p_i' \in \mathcal M$, $p_i^T p_i' = -2/3$ ($i = 1,2$).
The condition $t^T t' = 2$ now can be written as
$$p_1^T p_2 + (p_1')^T p_2' - p_1^T p_2' - (p_1')^T p_2 = 2.$$

Each scalar product equals $4/3$, $1/3$, or $-2/3$, with no more than one scalar product being equal to $4/3$ (i.e., no more than one pair of vectors with
different indices coincide). This leaves the following four possibilities (up to an interchange of pairs $(p_1, p_1')$ and $(p_2, p_2')$ or swapping
vectors within these two pairs simultaneously).
$$4/3 + (-2/3) - (-2/3) - (-2/3) = 2, \qquad 4/3 + 1/3 - 1/3 - (-2/3) = 2,$$
$$4/3 + 1/3 - (-2/3) - 1/3 = 2, \qquad 1/3 + 1/3 - (-2/3) - (-2/3) = 2.$$

The first three identities imply $p_1 = p_2$. Also, without loss of generality we can asume that $p_1 = a_1$ and $p_1' = b_2$, as all other cases are
similar up to some action of $W(\Esix)$. Then the distance from $p_2'$ to both $a_1$ and $b_2$ should equal $\sqrt{2}$, which leaves the only case
$p_2' = c_{12}$, giving the planar sets of type (c).

If the fourth identity emerges, we can also assume $p_1 = a_1$ and $p_1' = b_2$. Then 
$$p_2 \in \{a_3, a_4, a_5, a_6, c_{23}, c_{24}, c_{25}, c_{26}\},$$
$$p_2' \in \{b_3, b_4, b_5, b_6, c_{13}, c_{14}, c_{15}, c_{16}\}.$$

Up to reassignment of indices, this results in the following pairs for $(p_2, p_2')$: $(a_3, b_4)$, $(c_{23}, c_{14})$, $(b_3, c_{13})$, or $(a_3, c_{13})$.
In all these cases we have a planar set of type (d).

\noindent{\bf Subcase 1.3.} In this case we write $t = p_1 - p_1'$, $t' = p_2 - p_2'$, where $p_i, p_i' \in \mathcal M$, $p_i^T p_i' = -2/3$ ($i = 1,2$).
The condition $t^T t' = 3$ can be rewritten as
$$p_1^T p_2 + (p_1')^T p_2' - p_1^T p_2' - (p_1')^T p_2 = 3.$$

Again, no more than one scalar product is equal to $4/3$, which results in the only possible case
$$4/3 + 1/3 - (-2/3) - (-2/3) = 3.$$

This leads to $p_1 = p_2$. Similarly to Subcase 1.2, we can assume $p_1 = p_2 = a_1$ and $p_1' = b_2$. Therefore
$$p_2' \in \{b_3, b_4, b_5, b_6, c_{13}, c_{14}, c_{15}, c_{16}\}.$$
Each option results in a planar set of type (e).

\noindent{\bf Case 2.} The Delaunay tessellation $\mathcal D(\Esix \cap \alpha)$ is rectangular. Let $\Xi$ be one of the rectangles of 
$\mathcal D(\Esix \cap \alpha)$.

Again, the set $(\mathcal R \cup \mathcal T) \cap \alpha$ contains at least 3 pairs of antipodal vectors. Therefore at least one of the
diagonal vectors of $\Xi$ is a minimal vector for $\Esix$.

Consequently, the sphere in $\mathbb R^6$, that has the diagonal of $\Xi$ as a diameter, is empty for $\Esix$. Hence $\Xi$ is a planar section of
a centrally symmetric Delaunay polytope from $\mathcal D(\Esix)$. This polytope has at least 4 vertices, so it is not a segment. Therefore it
is a crosspolytope. But all planar sections of such crosspolytopes are equivalent up to automorphisms of $\Esix$, so only one type of a planar set
is possible, namely, the type (a).
 
\end{proof}

\begin{proof}[Proof of Lemma~\ref{lem:transversal_to_m_edge}]
Define
\begin{equation}\label{eq:facets}
\begin{array}{c}
F = \conv \{a_1, a_2, a_3, a_4, a_5, a_6, -b_1, -b_2, -b_3, -b_4, -b_5, -b_6 \},\\
F' = \conv \{a_2, b_2, c_{13}, c_{14}, c_{15}, c_{16}, -a_1, -b_1, -c_{23}, -c_{24}, -c_{25}, -c_{26}\}.
\end{array}
\end{equation}

$F$ and $F'$ are facets of $P_V(\Esix)$, and their facet vectors are $q = a_1 - b_1$ and $q' = a_2 - a_1$ respectively. The edge $[a_2, -b_1]$
is standard for $P_V(\Esix)$ with the standard vector $t = a_2 - b_1$. It is easy to see that $t = q + q'$.

Suppose that $[a_2, -b_1] \oplus z(p)$ is a standard 2-face of the parallelotope $P_V(\Esix) + z(p)$.

Let $A_p$ be the operator moving the lattice $\Lambda(P_V(\Esix))$ to the lattice $\Lambda(P_V(\Esix) + z(p))$ as defined in Lemma~\ref{lem:operator}. 
One can see that $t$ is a standard vector of the parallelotope $P_V(\Esix) + z(p)$, passing through the center of the face $[a_2, -b_1] \oplus z(p)$. 
Therefore $A_p t = t$. Further, $A_p q$ and $A_p q'$ are facet vectors of the parallelotope $P_V(\Esix) + z(p)$, because facet vectors of $P_V(\Esix)$
always correspond to facet vectors of $P_V(\Esix) + z(p)$. Hence
$$A_p q + A_p q' = t = q + q'.$$

Recall that the operator $A_p$ has the form $A_p = Id + 2p e_p^T$, where $e_p$ is a particularly chosen normal to the layer defined by $p$.
Hence $e_p^T q = -e_p^T q'$.

If $e_p^T q = -e_p^T q' = 0$, then $p$ is parallel to both $F$ and $F'$, which results in 
$$p \in \{ c_{12}, c_{34}, c_{35}, c_{36}, c_{45}, c_{46}, c_{56} \}.$$
However, $p = c_{12}$ is not a possible option because $z(c_{12})$ and $[a_2, -b_1]$ are parallel and do not form a direct Minkowski sum.
But if $p \in \{ c_{34}, c_{35}, c_{36}, c_{45}, c_{46}, c_{56} \}$, then $q = A_p q$ and $q' = A_p q'$ are lattice vectors of the lattice
$\Lambda(P_V(\Esix) + z(p))$. Therefore $t = q + q'$ is also a lattice vector of $\Lambda(P_V(\Esix) + z(p))$. Consequently, $P_V(\Esix) + t + z(p)$
belongs to the same tiling as $P_V(\Esix) + z(p)$, and
$$(P_V(\Esix) + z(p)) \cap (P_V(\Esix) + t + z(p)) = [a_2, -b_1] + z(p).$$
The sum in the right-hand side is direct, because the segments $[a_2, -b_1]$ and $z(p)$ are not parallel. Hence $p$ is strongly transversal to $[a_2, -b_1]$.

Now suppose that $e_p^T q \neq 0$. Then the scalar products $e_p^T q$ and $e_p^T q'$ have opposite signs. Therefore $p$ is parallel neither to $F$ nor to 
$F'$. Moreover, if the vector $p$ intersects the facet $F$ inwards $P_V(\Esix)$, then $p$ intersects the facet $F'$ outwards $P_V(\Esix)$, and vice versa. 
Thus scalar products of $p$ with the facet normals of $F$ and $F'$, i.e., with $q$ and $q'$, should have opposite signs.
Among all vectors of $\mathcal M$ this property holds only for $a_1$ and $b_2$.

If $p = a_1$, then $p$ intersects $F$ in outer direction, therefore $A_p q = q + 2p$ and $A_p q' =  q' - 2p$. Similarly, if $p = b_2$,
$A_p q = q - 2p$ and $A_p q' =  q' + 2p$. Thus $t = A_p q + A_p q'$ is a lattice vector of $\Lambda(P_V(\Esix) + z(p))$ as well.

Again, this means that $P_V(\Esix) + t + z(p)$ belongs to the same tiling as $P_V(\Esix) + z(p)$, and
$$(P_V(\Esix) + z(p)) \cap (P_V(\Esix) + t + z(p)) = [a_2, -b_1] + z(p).$$
Similarly to the previous case, the sum in the right-hand side is direct, and thus $p$ is strongly transversal to $[a_2, -b_1]$.
\end{proof}

\begin{proof}[Proof of Lemma~\ref{lem:six_belts}]
We prove all the assertions independently.

\noindent{\bf 1.} Indeed, $P_V(\Esix)$ already has a 6-belt with facet vectors $\{\pm(a_2 - a_1), \pm(a_3 - a_1), \pm (a_3 - a_2)\}$. 
This $6$-belts turns into a $6$-belt of $P_V(\Esix) + Z(U)$ with the same direction of the generating $(n-2)$-face.

\noindent{\bf 2.} The vectors $\pm(a_2 - a_1)$ and $\pm(b_1 - a_1)$ are already facet vectors of $P_V(\Esix)$. Their directions,
$(a_2 - a_1)^{\bot}$ and $(b_1 - a_1)^{\bot}$ do not change after adding a zonotope $Z(U)$. 

Now consider the 6-belt of $P_V(\Esix) + Z(U)$ with facet vectors corresponding to $\{\pm(a_2 - a_1), \pm(b_1 - a_1), \pm(b_2 - a_1)\}$.
Its generating face $G$ should be parallel to both $(a_2 - a_1)^{\bot}$ and $(b_1 - a_1)^{\bot}$, so
$$G \parallel \left\langle a_2 - a_1, b_1 - a_1\right\rangle^{\bot}.$$
We also notice that 
$$\mathcal M \cap \left\langle a_2 - a_1, b_1 - a_1\right\rangle^{\bot} = \{c_{12}, c_{34}, c_{56}, c_{35}, c_{46}, c_{36}, c_{45} \}.$$

Consider the facets $F$ and $F'$ of $P_V(\Esix)$ as in formula (\ref{eq:facets}).
Recall that their facet vectors are $a_1 - b_1$ and $a_2 - a_1$ respectively. Let $F + Z(U_1)$ and $F' + Z(U'_1)$ be the
corresponding facets of $P_V(\Esix) + Z(U)$. $F + Z(U_1)$ and $F' + Z(U'_1)$ are adjacent, because in the set 
$\{\pm(a_2 - a_1), \pm(b_1 - a_1), \pm(b_2 - a_1)\}$ the vectors $a_1 - b_1$ and $a_2 - a_1$ are neighbouring, i.e.
no other vector from this 6-tuple lies is a positive combination of these two.

The direction $\left\langle a_2 - a_1, b_1 - a_1\right\rangle^{\bot}$ generates a belt of $P_V(\Esix) + Z(U)$. Therefore the intersection 
$$F + Z(U_1) \cap F' + Z(U'_1)$$
should be a $4$-dimensional face parallel to $\left\langle a_2 - a_1, b_1 - a_1\right\rangle^{\bot}$. 

However, $F \cap F' = [a_2, -b_1]$, i.e., $1$-dimensional. To expand the intersection to a 4-dimensional polytope, the set $U$ 
should contain 3 vectors, making together with $-b_1 - a_2 = c_{12}$ a 4-dimensional linearly independent set. This can only be achieved
by taking one vector from each pair
\begin{equation}\label{eq:pairs}
(c_{34}, c_{56}), \quad (c_{35}, c_{46}), \quad (c_{36}, c_{45}).
\end{equation}

On the other hand, let $U_2$ be a $3$-element set obtained by an arbitrary choice of representatives from each pair in~\eqref{eq:pairs}.
Then $P_V(\Esix) + Z(U_2)$ has a $4$-belt generated by a $4$-face parallel to $\left\langle a_2 - a_1, b_1 - a_1\right\rangle^{\bot}$. 

The parallelotope $P_V(\Esix) + Z(U_2)$ has a 4-face $G = [b_2, -a_1] + Z(U_2)$ which is standard. The standard vector of $G$ corresponds
to $b_2 - a_1$. Therefore extending $P_V(\Esix) + Z(U_2)$ to $P_V(\Esix) + Z(U)$ should turn $G$ into a facet. For this purpose it is
necessary and sufficient for $U$ to contain a vector $p_4$ strongly transversal to $G$.

Therefore $p_4$ is strongly transversal to the edge $[b_2, -a_1]$ of $P$, but is not parallel to $G$. By Lemma~\ref{lem:transversal_to_m_edge},
this is possible only for $p_4 = a_2$ or $p_4 = b_1$.

All the arguments are reversible. Indeed, if $U$ contains $U_2 = \{p_1, p_2, p_3\}$ as a subset, then $P_V(\Esix) + Z(U_2)$
has a standard 4-face $G = [b_2, -a_1] + Z(U_2)$ whose standard vector is exactly $t = b_2 - a_1$. Similarly to the argument of 
Lemma~\ref{lem:transversal_to_m_edge}, $t$ is also a lattice vector of $\Lambda (P_V(\Esix) + Z(U_2) + z(p_4))$. Thus
$$ (P_V(\Esix) + Z(U_2) + z(p_4)) \cap (P_V(\Esix) + Z(U_2) + z(p_4) + t) = G + z(p_4), $$
which is a 5-polytope. 

Hence the 4-belt of $P_V(\Esix) + Z(U_2)$ has turned into a 6-belt of $P_V(\Esix) + Z(U_2) + z(p_4)$, 
which persists in $P + Z(U)$. Assertion 2 is proved completely.

\noindent{\bf 3.} Let the $4$-face $G$ generate the $6$-belt of $P_V(\Esix) + Z(U)$, whose facet vectors correspond to 
$$\{\pm(b_2 - a_1), \pm(b_3 - a_1), \pm(b_3 - b_2)\}.$$ 

First we prove that the affine hull of $G$ is spanned by 4 vectors of $\mathcal M$. 

Indeed, $P_V(\Esix)$ has the $m$-edge $[b_2, -a_1]$ with standard vector $t = b_2 - a_1$. The $6$-belt from the condition of Assertion 3 
has a facet  $[b_2, -a_1] + Z(U_1)$ for some $U_1\subset U$. All zone vectors of the zonotope $[b_2, -a_1] + Z(U_1)$ belong to $\mathcal M$, 
therefore every subface of $[b_2, -a_1] + Z(U_1)$ is also a zonotope with all zone vectors in $\mathcal M$. In particular, this holds for $G$.

Put
$$F = \conv\{ b_3, a_3, c_{12}, c_{24}, c_{25}, c_{26}, -b_2, -a_2, -c_{13}, -c_{34}, -c_{35}, -c_{36} \},$$
i.e., $F$ is the facet of $P_V(\Esix)$ with the facet vector $b_3 - b_2$. Let $U_2 \subset U_1 \cup \{c_{12}\}$ be the $4$-element vector set spanning the linear hull of $G$. $G$ is parallel to $F$, so $q(t) \notin U_2$. Thus $U_2\subset U_1$ and $P_V(\Esix) + Z(U_2)$ is a parallelotope.

But $P_V(\Esix) + Z(U_2)$ has nonzero width in the direction $\langle U_2 \rangle$. Thus it has a belt parallel to $\langle U_2 \rangle$. The facet vectors
of this belt cannot correspond to other vectors than $\{\pm(b_2 - a_1), \pm(b_3 - a_1), \pm(b_3 - b_2)\}$, because they persist in the belt of 
$P_V(\Esix) + Z(U)$ parallel to $G$.

The belt of $P_V(\Esix) + Z(U_2)$ parallel to $\langle U_2 \rangle$ should be a $6$-belt, because every $4$-belt requires 8 coplanar standard vectors, but the plane 
$$\langle b_2 - a_1, b_3 - a_1, b_3 - b_2 \rangle$$ 
contains only six. Hence $P_V(\Esix) + Z(U_2)$ has facets that have expanded from the edges $[b_2, -a_1]$ and $[b_3, -a_1]$. 
The only candidates for these facets are $[b_2, -a_1] + Z(U_2)$ and $[b_3, -a_1] + Z(U_2)$ respectively, because a 
$5$-dimensional zonotope has at least 5 zone vectors. By Lemma 3, if a standard face expanded to a standard face, all intermediate expansions are
also standard. As a result, for every $p \in U_2$ the 2-faces $[b_2, -a_1] \oplus z(p)$ and $[b_3, -a_1] \oplus z(p)$ are standard for $P_V(\Esix) + z(p)$. 
Applying Lemma~\ref{lem:transversal_to_m_edge} and transforms from $W(\Esix)$, that send the segment $[a_2, -b_1]$ to $[b_2, -a_1]$ and 
$[b_3, -a_1]$ respectively, gives
\begin{multline*}
U_2 \subset \{b_1, a_2, c_{34}, c_{35}, c_{36}, c_{45}, c_{46}, c_{56} \} \cap \\
\{b_1, a_3, c_{24}, c_{25}, c_{26}, c_{45}, c_{46}, c_{56} \} = \{b_1, c_{45}, c_{46}, c_{56} \}.
\end{multline*}

Since $|U_2| = 4$, $U_2 = \{b_1, c_{45}, c_{46}, c_{56} \}$.

However, the polytope $P_V(\Esix) + Z(U_2)$ has a belt consisting of facets $\pm (F + Z(U_2))$, $\pm ([b_2, -a_1] + Z(U_2))$, and $\pm ([b_2, -a_1] + Z(U_2))$,
and this belt remains in all polytopes $P_V(\Esix) + Z(U)$, where $U\supseteq U_2$, if only $P_V(\Esix) + Z(U)$ is a parallelotope.

\noindent{\bf 4.} The facet, whose facet vectors corresponds to the vector $a_2 - a_1$, is orthogonal to $a_2 - a_1$. Further,
the facet, whose facet vectors corresponds to the vector $b_1 - a_2$, is orthogonal to $b_1 - a_2$. Indeed, if a segment
$z(p)$, $p \in \mathcal M$, is strongly transversal to the edge $[b_1, -a_2]$, then, by Lemma~\ref{lem:transversal_to_m_edge}, $p$ is orthogonal 
to $b_1 - a_2$. Therefore the facet obtained by extensions of the edge edge $[b_1, -a_2]$ is orthogonal to $b_1 - a_2$.

Assume that the 6-belt with facet vectors corresponding to
$$\{\pm(a_2 - a_1), \pm(b_1 - a_2), \pm(b_2 - a_1)\}$$
exists. Then it is generated by a 4-face parallel to 
$$\left\langle a_2 - a_1, b_1 - a_2 \right\rangle^\bot = \left\langle a_2 - a_1, b_1 - a_1\right\rangle^\bot.$$

But a belt with this direction of the generating 4-face should contain two facets parallel to $\left\langle b_1 - a_1 \right\rangle^\bot$,
because such facets are already in $P_V(\Esix)$. The facet vectors of these facets correspond to $\pm(b_1 - a_1)$. But by assumption, there are
no such facet vectors in the 6-belt we consider. The contradiction finishes the proof of Assertion~4.

\noindent{\bf 5.} Similarly to the proof of Assertion 3 we can establish that there is a $4$-element set $U_2 \subset U$
such that $P_V(\Esix) + Z(U_2)$ has a $6$-belt consisting of facets $\pm([a_1, -b_2] + Z(U_2))$, $\pm([a_1, -c_{12}] + Z(U_2))$ and
$\pm([b_2, -c_{12}] + Z(U_2))$. This means that for every $p \in U_2$ the 2-faces $[a_1, -b_2] \oplus z(p)$, $[a_1, -c_{12}] \oplus z(p)$, 
and $[b_2, -c_{12}] \oplus z(p)$ are standard for $P_V(\Esix) + z(p)$. As a result,
\begin{multline*}
U_2 \subset \{a_2, b_1, c_{34}, c_{35}, c_{36}, c_{45}, c_{46}, c_{56} \} \cap \\
\{a_3, a_4, a_5, a_6, c_{23}, c_{24}, c_{25}, c_{26} \} \cap\\
\{b_3, b_4, b_5, b_6, c_{13}, c_{14}, c_{15}, c_{16} \} = \varnothing.
\end{multline*}

A contradiction proves that no $6$-belt can appear as in the condition of Assertion~5.

\noindent{\bf 6.} Assume that a parallelotope $P_V(\Esix) + Z(U)$ has a $6$-belt with facet vectors corresponding to 
$\{\pm(b_2 - a_1), \pm(b_4 - a_3), \pm(c_{23} - c_{14})\}$. Then the facets of such a 6-belt are extensions of edges $\pm [a_1, -b_2]$,
$\pm [a_3 - b_4]$, and $\pm [c_{23}, -c_{14}]$. But these edges are pairwise disjoint, then so are their extensions.

Indeed, let $G_1$ and $G_2$ be faces of $P_V(\Esix)$ such that $G_1 + Z(U_1)$ and $G_2 + Z(U_2)$ are faces of $P_V(\Esix) + Z(U)$ satisfying
$$(G_1 + Z(U_1)) \cap (G_2 + Z(U_2)) \neq \varnothing.$$
Then one can see that the intersection is of the form 
$$(G_1 + Z(U_1)) \cap (G_2 + Z(U_2)) = G_3 + Z(U_3)$$ 
with $G_3 \subset G_1 \cap G_2$.

Hence a contradiction, because facets making together a belt cannot be pairwise disjoint.

\end{proof}

\end{document}

%% file: K12_e.pstex_t
\begin{picture}(0,0)%
\includegraphics{K12_e.pstex}%
\end{picture}%
\setlength{\unitlength}{1973sp}%
\begingroup\makeatletter\ifx\SetFigFont\undefined%
\gdef\SetFigFont#1#2#3#4#5{%
  \reset@font\fontsize{#1}{#2pt}%
  \fontfamily{#3}\fontseries{#4}\fontshape{#5}%
  \selectfont}%
\fi\endgroup%
\begin{picture}(3016,1666)(68,-1344)
\put(1651,-586){\makebox(0,0)[lb]{\smash{{\SetFigFont{6}{7.2}{\familydefault}{\mddefault}{\updefault}{\color[rgb]{0,0,0}$c_{12}$}%
}}}}
\put(1126,-361){\makebox(0,0)[lb]{\smash{{\SetFigFont{6}{7.2}{\rmdefault}{\mddefault}{\updefault}{\color[rgb]{0,0,0}$a_{2}$}%
}}}}
\put(2101,-361){\makebox(0,0)[lb]{\smash{{\SetFigFont{6}{7.2}{\rmdefault}{\mddefault}{\updefault}{\color[rgb]{0,0,0}$a_{1}$}%
}}}}
\put(2026, 89){\makebox(0,0)[lb]{\smash{{\SetFigFont{6}{7.2}{\rmdefault}{\mddefault}{\updefault}{\color[rgb]{0,0,0}$c_{34}$}%
}}}}
\put(1126,-811){\makebox(0,0)[lb]{\smash{{\SetFigFont{6}{7.2}{\rmdefault}{\mddefault}{\updefault}{\color[rgb]{0,0,0}$b_{1}$}%
}}}}
\put(1951,-1186){\makebox(0,0)[lb]{\smash{{\SetFigFont{6}{7.2}{\rmdefault}{\mddefault}{\updefault}{\color[rgb]{0,0,0}$c_{56}$}%
}}}}
\put(526,-1036){\makebox(0,0)[lb]{\smash{{\SetFigFont{6}{7.2}{\rmdefault}{\mddefault}{\updefault}{\color[rgb]{0,0,0}$c_{46}$}%
}}}}
\put(2026,-886){\makebox(0,0)[lb]{\smash{{\SetFigFont{6}{7.2}{\rmdefault}{\mddefault}{\updefault}{\color[rgb]{0,0,0}$b_{2}$}%
}}}}
\put(526,-61){\makebox(0,0)[lb]{\smash{{\SetFigFont{6}{7.2}{\rmdefault}{\mddefault}{\updefault}{\color[rgb]{0,0,0}$c_{35}$}%
}}}}
\end{picture}%

%% file: Ggamma.pstex_t
\begin{picture}(0,0)%
\includegraphics{Ggamma.pstex}%
\end{picture}%
\setlength{\unitlength}{1973sp}%
\begingroup\makeatletter\ifx\SetFigFont\undefined%
\gdef\SetFigFont#1#2#3#4#5{%
  \reset@font\fontsize{#1}{#2pt}%
  \fontfamily{#3}\fontseries{#4}\fontshape{#5}%
  \selectfont}%
\fi\endgroup%
\begin{picture}(2955,2928)(-239,-2605)
\put(526,-2086){\makebox(0,0)[lb]{\smash{{\SetFigFont{6}{7.2}{\rmdefault}{\mddefault}{\updefault}{\color[rgb]{0,0,0}$c_{34}$}%
}}}}
\put(2026,-1186){\makebox(0,0)[lb]{\smash{{\SetFigFont{6}{7.2}{\rmdefault}{\mddefault}{\updefault}{\color[rgb]{0,0,0}$a_{2}$}%
}}}}
\put(2701,-1186){\makebox(0,0)[lb]{\smash{{\SetFigFont{6}{7.2}{\rmdefault}{\mddefault}{\updefault}{\color[rgb]{0,0,0}$b_{2}$}%
}}}}
\put(826,-1186){\makebox(0,0)[lb]{\smash{{\SetFigFont{6}{7.2}{\rmdefault}{\mddefault}{\updefault}{\color[rgb]{0,0,0}$a_{3}$}%
}}}}
\put(-224,-1186){\makebox(0,0)[lb]{\smash{{\SetFigFont{6}{7.2}{\rmdefault}{\mddefault}{\updefault}{\color[rgb]{0,0,0}$b_{3}$}%
}}}}
\put(1951,-211){\makebox(0,0)[lb]{\smash{{\SetFigFont{6}{7.2}{\rmdefault}{\mddefault}{\updefault}{\color[rgb]{0,0,0}$c_{12}$}%
}}}}
\put(451,-211){\makebox(0,0)[lb]{\smash{{\SetFigFont{6}{7.2}{\rmdefault}{\mddefault}{\updefault}{\color[rgb]{0,0,0}$c_{13}$}%
}}}}
\put(1201,164){\makebox(0,0)[lb]{\smash{{\SetFigFont{6}{7.2}{\rmdefault}{\mddefault}{\updefault}{\color[rgb]{0,0,0}$a_{1}$}%
}}}}
\put(1201,-736){\makebox(0,0)[lb]{\smash{{\SetFigFont{6}{7.2}{\rmdefault}{\mddefault}{\updefault}{\color[rgb]{0,0,0}$b_{1}$}%
}}}}
\put(1276,-1636){\makebox(0,0)[lb]{\smash{{\SetFigFont{6}{7.2}{\rmdefault}{\mddefault}{\updefault}{\color[rgb]{0,0,0}$b_{4}$}%
}}}}
\put(1276,-2536){\makebox(0,0)[lb]{\smash{{\SetFigFont{6}{7.2}{\rmdefault}{\mddefault}{\updefault}{\color[rgb]{0,0,0}$a_{4}$}%
}}}}
\put(1876,-2086){\makebox(0,0)[lb]{\smash{{\SetFigFont{6}{7.2}{\rmdefault}{\mddefault}{\updefault}{\color[rgb]{0,0,0}$c_{24}$}%
}}}}
\end{picture}%

%% file: Ga.pstex_t
\begin{picture}(0,0)%
\includegraphics{Ga.pstex}%
\end{picture}%
\setlength{\unitlength}{1973sp}%
\begingroup\makeatletter\ifx\SetFigFont\undefined%
\gdef\SetFigFont#1#2#3#4#5{%
  \reset@font\fontsize{#1}{#2pt}%
  \fontfamily{#3}\fontseries{#4}\fontshape{#5}%
  \selectfont}%
\fi\endgroup%
\begin{picture}(3030,2928)(-389,-2530)
\put(1126,239){\makebox(0,0)[lb]{\smash{{\SetFigFont{6}{7.2}{\rmdefault}{\mddefault}{\updefault}{\color[rgb]{0,0,0}$a_{1}$}%
}}}}
\put(1126,-736){\makebox(0,0)[lb]{\smash{{\SetFigFont{6}{7.2}{\rmdefault}{\mddefault}{\updefault}{\color[rgb]{0,0,0}$b_{1}$}%
}}}}
\put(-374,-1186){\makebox(0,0)[lb]{\smash{{\SetFigFont{6}{7.2}{\rmdefault}{\mddefault}{\updefault}{\color[rgb]{0,0,0}$b_{3}$}%
}}}}
\put(1951,-1186){\makebox(0,0)[lb]{\smash{{\SetFigFont{6}{7.2}{\rmdefault}{\mddefault}{\updefault}{\color[rgb]{0,0,0}$a_{2}$}%
}}}}
\put(151,-1711){\makebox(0,0)[lb]{\smash{{\SetFigFont{6}{7.2}{\rmdefault}{\mddefault}{\updefault}{\color[rgb]{0,0,0}$c_{35}$}%
}}}}
\put(451,-211){\makebox(0,0)[lb]{\smash{{\SetFigFont{6}{7.2}{\rmdefault}{\mddefault}{\updefault}{\color[rgb]{0,0,0}$c_{13}$}%
}}}}
\put(1876,-211){\makebox(0,0)[lb]{\smash{{\SetFigFont{6}{7.2}{\rmdefault}{\mddefault}{\updefault}{\color[rgb]{0,0,0}$c_{12}$}%
}}}}
\put(826,-2011){\makebox(0,0)[lb]{\smash{{\SetFigFont{6}{7.2}{\rmdefault}{\mddefault}{\updefault}{\color[rgb]{0,0,0}$c_{45}$}%
}}}}
\put(676,-2461){\makebox(0,0)[lb]{\smash{{\SetFigFont{6}{7.2}{\rmdefault}{\mddefault}{\updefault}{\color[rgb]{0,0,0}$b_{4}$}%
}}}}
\put(2626,-1186){\makebox(0,0)[lb]{\smash{{\SetFigFont{6}{7.2}{\rmdefault}{\mddefault}{\updefault}{\color[rgb]{0,0,0}$b_{2}$}%
}}}}
\put(976,-1186){\makebox(0,0)[lb]{\smash{{\SetFigFont{6}{7.2}{\rmdefault}{\mddefault}{\updefault}{\color[rgb]{0,0,0}$a_{3}$}%
}}}}
\put(2176,-1936){\makebox(0,0)[lb]{\smash{{\SetFigFont{6}{7.2}{\rmdefault}{\mddefault}{\updefault}{\color[rgb]{0,0,0}$c_{24}$}%
}}}}
\put(1276,-1711){\makebox(0,0)[lb]{\smash{{\SetFigFont{6}{7.2}{\rmdefault}{\mddefault}{\updefault}{\color[rgb]{0,0,0}$a_{4}$}%
}}}}
\end{picture}%

%% file: Gb.pstex_t
\begin{picture}(0,0)%
\includegraphics{Gb.pstex}%
\end{picture}%
\setlength{\unitlength}{1973sp}%
\begingroup\makeatletter\ifx\SetFigFont\undefined%
\gdef\SetFigFont#1#2#3#4#5{%
  \reset@font\fontsize{#1}{#2pt}%
  \fontfamily{#3}\fontseries{#4}\fontshape{#5}%
  \selectfont}%
\fi\endgroup%
\begin{picture}(3466,3467)(-232,-2694)
\put(1126,-736){\makebox(0,0)[lb]{\smash{{\SetFigFont{6}{7.2}{\rmdefault}{\mddefault}{\updefault}{\color[rgb]{0,0,0}$b_{1}$}%
}}}}
\put(1051,-1636){\makebox(0,0)[lb]{\smash{{\SetFigFont{6}{7.2}{\rmdefault}{\mddefault}{\updefault}{\color[rgb]{0,0,0}$b_{4}$}%
}}}}
\put(826,-1111){\makebox(0,0)[lb]{\smash{{\SetFigFont{6}{7.2}{\rmdefault}{\mddefault}{\updefault}{\color[rgb]{0,0,0}$a_{3}$}%
}}}}
\put(1951,-1186){\makebox(0,0)[lb]{\smash{{\SetFigFont{6}{7.2}{\rmdefault}{\mddefault}{\updefault}{\color[rgb]{0,0,0}$a_{2}$}%
}}}}
\put(451,-136){\makebox(0,0)[lb]{\smash{{\SetFigFont{6}{7.2}{\rmdefault}{\mddefault}{\updefault}{\color[rgb]{0,0,0}$c_{13}$}%
}}}}
\put(1501,-2386){\makebox(0,0)[lb]{\smash{{\SetFigFont{6}{7.2}{\rmdefault}{\mddefault}{\updefault}{\color[rgb]{0,0,0}$a_{4}$}%
}}}}
\put(2326,-2011){\makebox(0,0)[lb]{\smash{{\SetFigFont{6}{7.2}{\rmdefault}{\mddefault}{\updefault}{\color[rgb]{0,0,0}$c_{24}$}%
}}}}
\put(751,614){\makebox(0,0)[lb]{\smash{{\SetFigFont{6}{7.2}{\rmdefault}{\mddefault}{\updefault}{\color[rgb]{0,0,0}$a_{1}$}%
}}}}
\put(2626,-211){\makebox(0,0)[lb]{\smash{{\SetFigFont{6}{7.2}{\rmdefault}{\mddefault}{\updefault}{\color[rgb]{0,0,0}$c_{46}$}%
}}}}
\put(1876,-286){\makebox(0,0)[lb]{\smash{{\SetFigFont{6}{7.2}{\rmdefault}{\mddefault}{\updefault}{\color[rgb]{0,0,0}$c_{12}$}%
}}}}
\put(3076,-1111){\makebox(0,0)[lb]{\smash{{\SetFigFont{6}{7.2}{\rmdefault}{\mddefault}{\updefault}{\color[rgb]{0,0,0}$b_{2}$}%
}}}}
\put(-74,-1036){\makebox(0,0)[lb]{\smash{{\SetFigFont{6}{7.2}{\rmdefault}{\mddefault}{\updefault}{\color[rgb]{0,0,0}$b_{3}$}%
}}}}
\put(2101,389){\makebox(0,0)[lb]{\smash{{\SetFigFont{6}{7.2}{\rmdefault}{\mddefault}{\updefault}{\color[rgb]{0,0,0}$c_{35}$}%
}}}}
\end{picture}%

%% file: Case7_G7.pstex_t
\begin{picture}(0,0)%
\includegraphics{Case7_G7.pstex}%
\end{picture}%
\setlength{\unitlength}{1973sp}%
\begingroup\makeatletter\ifx\SetFigFont\undefined%
\gdef\SetFigFont#1#2#3#4#5{%
  \reset@font\fontsize{#1}{#2pt}%
  \fontfamily{#3}\fontseries{#4}\fontshape{#5}%
  \selectfont}%
\fi\endgroup%
\begin{picture}(3780,3828)(-389,-2455)
\put(1426,-2386){\makebox(0,0)[lb]{\smash{{\SetFigFont{6}{7.2}{\rmdefault}{\mddefault}{\updefault}{\color[rgb]{0,0,0}$b_1$}%
}}}}
\put(2851,-1936){\makebox(0,0)[lb]{\smash{{\SetFigFont{6}{7.2}{\rmdefault}{\mddefault}{\updefault}{\color[rgb]{0,0,0}$a_4$}%
}}}}
\put(-374,-661){\makebox(0,0)[lb]{\smash{{\SetFigFont{6}{7.2}{\rmdefault}{\mddefault}{\updefault}{\color[rgb]{0,0,0}$b_5$}%
}}}}
\put(226,-1861){\makebox(0,0)[lb]{\smash{{\SetFigFont{6}{7.2}{\rmdefault}{\mddefault}{\updefault}{\color[rgb]{0,0,0}$a_2$}%
}}}}
\put(151,689){\makebox(0,0)[lb]{\smash{{\SetFigFont{6}{7.2}{\rmdefault}{\mddefault}{\updefault}{\color[rgb]{0,0,0}$a_3$}%
}}}}
\put(1501,1214){\makebox(0,0)[lb]{\smash{{\SetFigFont{6}{7.2}{\rmdefault}{\mddefault}{\updefault}{\color[rgb]{0,0,0}$a_1$}%
}}}}
\put(2851,689){\makebox(0,0)[lb]{\smash{{\SetFigFont{6}{7.2}{\rmdefault}{\mddefault}{\updefault}{\color[rgb]{0,0,0}$b_2$}%
}}}}
\put(3376,-586){\makebox(0,0)[lb]{\smash{{\SetFigFont{6}{7.2}{\rmdefault}{\mddefault}{\updefault}{\color[rgb]{0,0,0}$b_3$}%
}}}}
\put(451,-1111){\makebox(0,0)[lb]{\smash{{\SetFigFont{6}{7.2}{\rmdefault}{\mddefault}{\updefault}{\color[rgb]{0,0,0}$c_{25}$}%
}}}}
\put(376, 14){\makebox(0,0)[lb]{\smash{{\SetFigFont{6}{7.2}{\rmdefault}{\mddefault}{\updefault}{\color[rgb]{0,0,0}$c_{35}$}%
}}}}
\put(2176,-1636){\makebox(0,0)[lb]{\smash{{\SetFigFont{6}{7.2}{\rmdefault}{\mddefault}{\updefault}{\color[rgb]{0,0,0}$c_{14}$}%
}}}}
\put(1951,614){\makebox(0,0)[lb]{\smash{{\SetFigFont{6}{7.2}{\rmdefault}{\mddefault}{\updefault}{\color[rgb]{0,0,0}$c_{12}$}%
}}}}
\put(2401,-736){\makebox(0,0)[lb]{\smash{{\SetFigFont{6}{7.2}{\rmdefault}{\mddefault}{\updefault}{\color[rgb]{0,0,0}$c_{34}$}%
}}}}
\end{picture}%

%% file: Case8_G1.pstex_t
\begin{picture}(0,0)%
\includegraphics{Case8_G1.pstex}%
\end{picture}%
\setlength{\unitlength}{1973sp}%
\begingroup\makeatletter\ifx\SetFigFont\undefined%
\gdef\SetFigFont#1#2#3#4#5{%
  \reset@font\fontsize{#1}{#2pt}%
  \fontfamily{#3}\fontseries{#4}\fontshape{#5}%
  \selectfont}%
\fi\endgroup%
\begin{picture}(4230,3978)(-839,-2155)
\put(2551,-1711){\makebox(0,0)[lb]{\smash{{\SetFigFont{6}{7.2}{\rmdefault}{\mddefault}{\updefault}{\color[rgb]{0,0,0}$a_2$}%
}}}}
\put(1426,-2086){\makebox(0,0)[lb]{\smash{{\SetFigFont{6}{7.2}{\rmdefault}{\mddefault}{\updefault}{\color[rgb]{0,0,0}$v_4$}%
}}}}
\put(376,-1861){\makebox(0,0)[lb]{\smash{{\SetFigFont{6}{7.2}{\rmdefault}{\mddefault}{\updefault}{\color[rgb]{0,0,0}$a_4$}%
}}}}
\put(2401,314){\makebox(0,0)[lb]{\smash{{\SetFigFont{6}{7.2}{\rmdefault}{\mddefault}{\updefault}{\color[rgb]{0,0,0}$c_{46}$}%
}}}}
\put(2251,1664){\makebox(0,0)[lb]{\smash{{\SetFigFont{6}{7.2}{\rmdefault}{\mddefault}{\updefault}{\color[rgb]{0,0,0}$c_{25}$}%
}}}}
\put(376,614){\makebox(0,0)[lb]{\smash{{\SetFigFont{6}{7.2}{\rmdefault}{\mddefault}{\updefault}{\color[rgb]{0,0,0}$a_1$}%
}}}}
\put(2176,-61){\makebox(0,0)[lb]{\smash{{\SetFigFont{6}{7.2}{\rmdefault}{\mddefault}{\updefault}{\color[rgb]{0,0,0}$a_3$}%
}}}}
\put(-824,-286){\makebox(0,0)[lb]{\smash{{\SetFigFont{6}{7.2}{\rmdefault}{\mddefault}{\updefault}{\color[rgb]{0,0,0}$b_4$}%
}}}}
\put(151,1514){\makebox(0,0)[lb]{\smash{{\SetFigFont{6}{7.2}{\rmdefault}{\mddefault}{\updefault}{\color[rgb]{0,0,0}$b_3$}%
}}}}
\put(301,-961){\makebox(0,0)[lb]{\smash{{\SetFigFont{6}{7.2}{\rmdefault}{\mddefault}{\updefault}{\color[rgb]{0,0,0}$c_{l4}$}%
}}}}
\put(3376,-286){\makebox(0,0)[lb]{\smash{{\SetFigFont{6}{7.2}{\rmdefault}{\mddefault}{\updefault}{\color[rgb]{0,0,0}$c_{35}$}%
}}}}
\put(1126,764){\makebox(0,0)[lb]{\smash{{\SetFigFont{6}{7.2}{\rmdefault}{\mddefault}{\updefault}{\color[rgb]{0,0,0}$c_{13}$}%
}}}}
\put(2476,-586){\makebox(0,0)[lb]{\smash{{\SetFigFont{6}{7.2}{\rmdefault}{\mddefault}{\updefault}{\color[rgb]{0,0,0}$c_{l2}$}%
}}}}
\put(1276,-1261){\makebox(0,0)[lb]{\smash{{\SetFigFont{6}{7.2}{\rmdefault}{\mddefault}{\updefault}{\color[rgb]{0,0,0}$b_1$}%
}}}}
\put(226,-211){\makebox(0,0)[lb]{\smash{{\SetFigFont{6}{7.2}{\rmdefault}{\mddefault}{\updefault}{\color[rgb]{0,0,0}$b_2$}%
}}}}
\end{picture}%

%% file: Case9_G1.pstex_t
\begin{picture}(0,0)%
\includegraphics{Case9_G1.pstex}%
\end{picture}%
\setlength{\unitlength}{1973sp}%
\begingroup\makeatletter\ifx\SetFigFont\undefined%
\gdef\SetFigFont#1#2#3#4#5{%
  \reset@font\fontsize{#1}{#2pt}%
  \fontfamily{#3}\fontseries{#4}\fontshape{#5}%
  \selectfont}%
\fi\endgroup%
\begin{picture}(3780,3466)(-239,-2394)
\put(2701,-2161){\makebox(0,0)[lb]{\smash{{\SetFigFont{6}{7.2}{\rmdefault}{\mddefault}{\updefault}{\color[rgb]{0,0,0}$c_{26}$}%
}}}}
\put(-224,-736){\makebox(0,0)[lb]{\smash{{\SetFigFont{6}{7.2}{\rmdefault}{\mddefault}{\updefault}{\color[rgb]{0,0,0}$c_{12}$}%
}}}}
\put(1126,-886){\makebox(0,0)[lb]{\smash{{\SetFigFont{6}{7.2}{\rmdefault}{\mddefault}{\updefault}{\color[rgb]{0,0,0}$a_2$}%
}}}}
\put(2251,-1711){\makebox(0,0)[lb]{\smash{{\SetFigFont{6}{7.2}{\rmdefault}{\mddefault}{\updefault}{\color[rgb]{0,0,0}$c_{15}$}%
}}}}
\put(2176,-886){\makebox(0,0)[lb]{\smash{{\SetFigFont{6}{7.2}{\rmdefault}{\mddefault}{\updefault}{\color[rgb]{0,0,0}$a_4$}%
}}}}
\put(3526,-661){\makebox(0,0)[lb]{\smash{{\SetFigFont{6}{7.2}{\rmdefault}{\mddefault}{\updefault}{\color[rgb]{0,0,0}$c_{34}$}%
}}}}
\put(451,-2011){\makebox(0,0)[lb]{\smash{{\SetFigFont{6}{7.2}{\rmdefault}{\mddefault}{\updefault}{\color[rgb]{0,0,0}$c_{46}$}%
}}}}
\put(901,-1561){\makebox(0,0)[lb]{\smash{{\SetFigFont{6}{7.2}{\rmdefault}{\mddefault}{\updefault}{\color[rgb]{0,0,0}$c_{35}$}%
}}}}
\put(1126,314){\makebox(0,0)[lb]{\smash{{\SetFigFont{6}{7.2}{\rmdefault}{\mddefault}{\updefault}{\color[rgb]{0,0,0}$b_2$}%
}}}}
\put(2176,389){\makebox(0,0)[lb]{\smash{{\SetFigFont{6}{7.2}{\rmdefault}{\mddefault}{\updefault}{\color[rgb]{0,0,0}$b_4$}%
}}}}
\put(1201,-361){\makebox(0,0)[lb]{\smash{{\SetFigFont{6}{7.2}{\rmdefault}{\mddefault}{\updefault}{\color[rgb]{0,0,0}$b_1$}%
}}}}
\put(2026,-361){\makebox(0,0)[lb]{\smash{{\SetFigFont{6}{7.2}{\rmdefault}{\mddefault}{\updefault}{\color[rgb]{0,0,0}$b_3$}%
}}}}
\put(2851,689){\makebox(0,0)[lb]{\smash{{\SetFigFont{6}{7.2}{\rmdefault}{\mddefault}{\updefault}{\color[rgb]{0,0,0}$a_3$}%
}}}}
\put(751,764){\makebox(0,0)[lb]{\smash{{\SetFigFont{6}{7.2}{\rmdefault}{\mddefault}{\updefault}{\color[rgb]{0,0,0}$a_1$}%
}}}}
\end{picture}%

%% file: Case10_G10.pstex_t
\begin{picture}(0,0)%
\includegraphics{Case10_G10.pstex}%
\end{picture}%
\setlength{\unitlength}{1973sp}%
\begingroup\makeatletter\ifx\SetFigFont\undefined%
\gdef\SetFigFont#1#2#3#4#5{%
  \reset@font\fontsize{#1}{#2pt}%
  \fontfamily{#3}\fontseries{#4}\fontshape{#5}%
  \selectfont}%
\fi\endgroup%
\begin{picture}(3023,3228)(661,-2605)
\put(2176,464){\makebox(0,0)[lb]{\smash{{\SetFigFont{6}{7.2}{\rmdefault}{\mddefault}{\updefault}{\color[rgb]{0,0,0}$a_2$}%
}}}}
\put(1051,-2236){\makebox(0,0)[lb]{\smash{{\SetFigFont{6}{7.2}{\rmdefault}{\mddefault}{\updefault}{\color[rgb]{0,0,0}$a_3$}%
}}}}
\put(2326,-1411){\makebox(0,0)[lb]{\smash{{\SetFigFont{6}{7.2}{\rmdefault}{\mddefault}{\updefault}{\color[rgb]{0,0,0}$a_4$}%
}}}}
\put(2026,-2536){\makebox(0,0)[lb]{\smash{{\SetFigFont{6}{7.2}{\rmdefault}{\mddefault}{\updefault}{\color[rgb]{0,0,0}$b_1$}%
}}}}
\put(1051,239){\makebox(0,0)[lb]{\smash{{\SetFigFont{6}{7.2}{\rmdefault}{\mddefault}{\updefault}{\color[rgb]{0,0,0}$b_3$}%
}}}}
\put(2026,-886){\makebox(0,0)[lb]{\smash{{\SetFigFont{6}{7.2}{\rmdefault}{\mddefault}{\updefault}{\color[rgb]{0,0,0}$c_{12}$}%
}}}}
\put(2251,-211){\makebox(0,0)[lb]{\smash{{\SetFigFont{6}{7.2}{\rmdefault}{\mddefault}{\updefault}{\color[rgb]{0,0,0}$c_{25}$}%
}}}}
\put(3601,-586){\makebox(0,0)[lb]{\smash{{\SetFigFont{6}{7.2}{\rmdefault}{\mddefault}{\updefault}{\color[rgb]{0,0,0}$c_{45}$}%
}}}}
\put(3301,-2236){\makebox(0,0)[lb]{\smash{{\SetFigFont{6}{7.2}{\rmdefault}{\mddefault}{\updefault}{\color[rgb]{0,0,0}$c_{14}$}%
}}}}
\put(751,-1411){\makebox(0,0)[lb]{\smash{{\SetFigFont{6}{7.2}{\rmdefault}{\mddefault}{\updefault}{\color[rgb]{0,0,0}$b_2$}%
}}}}
\put(676,-586){\makebox(0,0)[lb]{\smash{{\SetFigFont{6}{7.2}{\rmdefault}{\mddefault}{\updefault}{\color[rgb]{0,0,0}$a_1$}%
}}}}
\put(3301,164){\makebox(0,0)[lb]{\smash{{\SetFigFont{6}{7.2}{\rmdefault}{\mddefault}{\updefault}{\color[rgb]{0,0,0}$b_5$}%
}}}}
\put(3601,-1411){\makebox(0,0)[lb]{\smash{{\SetFigFont{6}{7.2}{\rmdefault}{\mddefault}{\updefault}{\color[rgb]{0,0,0}$c_{36}$}%
}}}}
\put(1801,-2161){\makebox(0,0)[lb]{\smash{{\SetFigFont{6}{7.2}{\rmdefault}{\mddefault}{\updefault}{\color[rgb]{0,0,0}$c_{13}$}%
}}}}
\put(1801,164){\makebox(0,0)[lb]{\smash{{\SetFigFont{6}{7.2}{\rmdefault}{\mddefault}{\updefault}{\color[rgb]{0,0,0}$c_{23}$}%
}}}}
\end{picture}%